\documentclass[reqno]{jnsao}
\usepackage[utf8]{inputenc}
\usepackage{amsfonts,amsthm}
\usepackage[nameinlink,capitalize]{cleveref}
\usepackage{algorithm}
\usepackage{algpseudocode}
\usepackage[most]{tcolorbox}
\usepackage{mathtools}
\usepackage[notref,notcite,final]{showkeys}
\usepackage{subcaption}
\newtcolorbox{breakablealgorithm}[1]{%
	enhanced,
	breakable,
	colback=white,
	colframe=white,      % no outer frame
	boxrule=0pt,         % no border
	sharp corners,
	top=2mm,
	bottom=2mm,
	left=0pt,
	right=0pt,
	before upper={
		\refstepcounter{algorithm}
		\hrule\vspace{0.5ex}
		\textbf{Algorithm \thealgorithm}\quad #1
		\vspace{0.5ex}\hrule\vspace{1ex}
	},
	after upper={
		\vspace{1ex}\hrule
	}
}

\theoremstyle{definition}
\newtheorem*{example*}{Example}
\newcommand{\alphasum}{\alpha_1 \oplus \alpha_2 - c}

\DeclareMathOperator{\diag}{diag}
\newcommand*{\dop}[1]{\,\mathop{{\mathrm{d}#1}}}
\newcommand{\drm}{\mathrm{d}}

\newcommand{\etasum}{\eta_1 \oplus \eta_2}
\newcommand*{\norm}[2][{}]{\|#2\|_{#1}}
\DeclareMathAlphabet{\Mathbb}{U}{bbold}{m}{n}
\newcommand{\one}{\Mathbb{1}}
\newcommand{\phantomtop}{\phantom{\top}}
\newcommand{\proj}{P}
\newcommand{\scalarproduct}[3][{}]{\langle #2, \, #3 \rangle_{#1}}

\newcommand{\StdSmplx}{\mathfrak{S}_{n_1}}
\newcommand{\substitute}{\boldsymbol{\cdot}}
\newcommand{\zero}{\Mathbb{0}}

\newcommand{\Acal}{\mathcal{A}}
\newcommand{\Bcal}{\mathcal{B}}
\newcommand{\Dcal}{\mathcal{D}}
\newcommand{\Fcal}{\mathcal{F}}
\newcommand{\Gcal}{\mathcal{G}}
\newcommand{\Jcal}{\mathcal{J}}
\newcommand{\Mcal}{\mathcal{M}}
\newcommand{\Ncal}{\mathcal{N}}
\newcommand{\Pcal}{\mathcal{P}}
\newcommand{\Scal}{\mathcal{S}}

\newcommand{\Bfrak}{\mathfrak{B}}
\newcommand{\Mfrak}{\mathfrak{M}}
\newcommand{\Pfrak}{\mathfrak{P}}

\newcommand{\N}{\mathbb{N}}

\newcommand{\R}{\mathbb{R}}

\newcommand*{\Borel}[1]{\Bfrak(#1)}
\newcommand*{\Coupling}[2]{\Pi(#1,#2)}
\newcommand*{\Measure}[1]{\Mfrak(#1)}
\newcommand*{\Probability}[1]{\Pfrak(#1)}

\newcommand{\eg}{e.g.\ }
\newcommand{\wrt}{w.r.t.\ }
\newcommand{\ie}{i.e.,\ }

\allowdisplaybreaks
\numberwithin{equation}{section}
\counterwithin{figure}{section}
\counterwithin{table}{section}
\newif\ifrevision
\revisionfalse

\let\oldtextcolor\textcolor
\renewcommand{\textcolor}[2]{%
	\ifrevision
	\oldtextcolor{#1}{#2}%
	\else
	#2%
	\fi
}

\manuscriptsubmitted{2024-06-14}
\manuscriptaccepted{2026-06-10}
\manuscriptvolume{6}
\manuscriptnumber{13780}
\manuscriptyear{2026}
\manuscriptdoi{10.46298/jnsao-2026-13780}
\manuscriptlicense{CC-BY-SA 4.0}

\title{Bilevel Optimization of the 
	Kantorovich Problem and its Quadratic Regularization \\
	Part III: The Finite-Dimensional Case}
\shorttitle{Bilevel Optimization of the Kantorovich Problem, Part III}

\author{Sebastian Hillbrecht\thanks{Sebastian Hillbrecht, Technische Universit\"at Dortmund, Fakult\"at f\"ur Mathematik, Lehrstuhl X, Vogelpothsweg 87, 44227 Dortmund, Germany, \email{sebastian.hillbrecht@tu-dortmund.de}, \orcid{0000-0003-4524-3420}.}}
\shortauthor{Hillbrecht}

\acknowledgements{This research was supported by the German Research Foundation (DFG) under grant number~LO 1436/9-1 within the priority program Non-smooth and Complementarity-based Distributed Parameter Systems: Simulation and Hierarchical Optimization (SPP~1962).}

\date{\ISOToday} 

\begin{document}

\maketitle

\begin{abstract}
	As the title suggests, this is the third paper in a series addressing bilevel optimization problems that are governed by the Kantorovich problem of optimal transport.
	These tasks can be reformulated as mathematical problems with complementarity constraints in the space of regular Borel measures. Due to the nonsmoothness that is introduced by the complementarity constraints, such problems are often regularized, for instance, using entropic regularization.
	In this series of papers, however, we apply a quadratic regularization to the Kantorovich problem. By doing so, we enhance its numerical properties while preserving the sparsity structure of the optimal transportation plan as much as possible.
    While the first two papers in this series focus on the well-posedness of the regularized bilevel problems and the approximation of solutions to the bilevel optimization problem in the infinite-dimensional case, in this paper, we reproduce these results for the finite-dimensional case and present findings that go well beyond the ones of the previous papers and pave the way for the numerical treatment of the bilevel problems.
\end{abstract}

% % % % % % % % % % % % % % % % % % % % % % % % % % % % % % % % % % % % % % %

\section{Introduction}\label{sc:INTRO}
The Kantorovich problem (of optimal transport) is given by
\begin{equation}
\label{pr:INTRO-Kant}
\tag{K}
	\begin{array}{rl}
		\inf\limits_{\pi} & \displaystyle  \int_{\Omega_1 \times \Omega_2} c \dop{\pi} \\[1em]
		\text{s.t.} & \pi \in \Coupling{\mu_1}{\mu_2}, \quad \pi \geq 0.
		%& \pi \in \Measure{\Omega_1 \times \Omega_2}, ~ \pi \geq 0, ~ {\proj_i}_{\#} \pi = \mu_i, ~ i = 1, 2.
	\end{array}
\end{equation}
In the above, $\Omega_1 \subset \R^{d_1}$ and $\Omega_2 \subset \R^{d_2}$ are (locally) compact domains, $c \colon \Omega_1 \times \Omega_2 \to \R$ is a (lower semi-)continuous cost function that is bounded from below, and
\begin{equation*}
	\Coupling{\mu_1}{\mu_2}
	\coloneqq \{
		\pi \in \Measure{\Omega_1 \times \Omega_2} \colon
		{\proj_1}_{\#} \pi = \mu_1, ~ {\proj_2}_{\#} \pi = \mu_2
	\}
\end{equation*}
denotes the set of transport plans between the source marginal $\mu_1 \in \Measure{\Omega_1}$ and the target marginal $\mu_2 \in \Measure{\Omega_2}$. Already Kantorovich himself knew that this problem is well-posed, see \cite{kantorovich1942translocation}.
%${\proj_i}_{\#} \pi$ denotes the pushforward measure of $\pi$ \wrt the projection $\proj_i \colon \Omega_1 \times \Omega_2 \ni (x_1, x_2) \mapsto x_i \in \Omega_i$, and, for $i = 1, 2$, the marginal $\mu_i \in \Measure{\Omega_i}$ is a nonnegative regular Borel measure on some locally compact domain $\Omega_i \subset \R^{d_i}$.
In the first two parts of this series of papers, \cite{hillbrecht2022bilevel1,hillbrecht2022bilevel2}, we treated the Kantorovich problem as the lower-level problem of a general bilevel optimization problem, \ie we investigated problems of the form
\begin{equation}
\label{pr:INTRO-BilevelKant}
\tag{BK}
	\begin{array}{rl}
		\inf\limits_{\pi, \mu_1} & \Jcal(\pi, \mu_1) \\[0.5em]
		\text{s.t.} & \pi \in \Measure{\Omega_1 \times \Omega_2}, \quad \mu_1 \in \Measure{\Omega_1}, \\[0.35em]
		& \mu_1 \geq 0, \quad \norm[\Measure{\Omega_1}]{\mu_1} = \norm[\Measure{\Omega_2}]{\mu_2^\drm}, \\[0.5em]
		& \pi ~ \text{solves \cref{pr:INTRO-Kant} w.r.t.} ~ \mu_1, ~ \mu_2^\drm, ~ \text{and} ~ c_\drm,
	\end{array}
\end{equation}
where $\Jcal \colon \Measure{\Omega_1 \times \Omega_2} \times \Measure{\Omega_1} \to \R \cup \{+\infty\}$ is an arbitrary weak$^*$ lower semicontinuous objective functional and $\mu_2^\drm \in \Measure{\Omega_2}$ \textcolor{cyan}{as well as} $c_\drm \in C(\Omega_1 \times \Omega_2)$ are given (and fixed) data. Depending on the choice of $\Jcal$, \cref{pr:INTRO-BilevelKant} is suited to model different tasks. For instance, if we consider a tracking-type objective of the form
\begin{equation*}
	\Jcal(\pi, \mu_1)
	= \norm[\Measure{D}]{\pi - \pi_\drm} + \norm[\Measure{D_1}]{\mu_1 - \mu_1^\drm},
\end{equation*}
where $D \subset \Omega_1 \times \Omega_2$ and $D_1 \subset \Omega_1$ are (open) observation domains and $\pi_\drm \in \Measure{D}$ \textcolor{cyan}{as well as} $\mu_1^\drm \in \Measure{D_1}$ are given data, then the bilevel problem \cref{pr:INTRO-BilevelKant} turns into the task of recovering an (unknown) transportation process from \textcolor{cyan}{perhaps} incomplete and/or noisy observations of the transport plan and the source marginal.

This is a special case of an inverse problem that is governed by the Kantorovich problem. While recovering certain properties of a transportation process from observations is not a new challenge, as a reference we only mention \cite{stuart2020inverse,chiu2022discrete,andrade2024sparsistency} and the references therein, to the author's best knowledge the above approach is unique as it involves the solution of a single optimization problem, whereas all of the mentioned approaches rely on Bayesian methods. 

Another possible application of the general bilevel problem \cref{pr:INTRO-BilevelKant} is given if one considers the compact domains $\Omega_1 = \Omega_2 \eqqcolon \Omega_* \subset \R^{d_*}$, the prior $\mu_2^\drm \in \Probability{\Omega_*}$, the cost $c_\drm(x, y) = \norm{x-y}^\rho$, $\rho > 1$, and the objective functional
\begin{equation*}
	\Jcal(\pi, \mu_1)
	= \frac12 \norm[Y]{G \mu_1 - y_\drm} + \nu \int_{\Omega_* \times \Omega_*} c_d \dop{\pi},
\end{equation*}
\textcolor{cyan}{in which} $G$ is a compact linear operator \textcolor{cyan}{that maps} the measure space $\Measure{\Omega_*}$ onto some Banach space $Y$, $y_\drm \in Y$ is a given observation, and $\nu > 0$ is a weighting parameter.

In this setting, \cref{pr:INTRO-BilevelKant} turns out to be equivalent to the \emph{Wasserstein(-regularized) inverse problem}
\begin{equation}
\label{pr:INTRO-WasserInvProb}
\tag{WI$_\rho$}
	\begin{array}{rl}
		\inf\limits_{\mu_1} & \frac12 \norm[Y]{G \mu_1 - y_\drm} + \nu W_\rho(\mu_1, \mu_2^\drm)^\rho \\[0.5em]
		\text{s.t.} & \mu_1 \in \Probability{\Omega_*},
	\end{array}
\end{equation}
where
\begin{equation*}
	W_\rho (\mu_1, \mu_2)
	\coloneqq \min_{\substack{\theta \in \Coupling{\mu_1}{\mu_2}, \\ \theta \geq 0}} \left( \int_{\Omega_* \times \Omega_*} \norm{x - y}^\rho \dop{\theta(x, y)} \right)^{\frac1\rho}
\end{equation*}
denotes the well-known $\rho$-Wasserstein distance between two marginals $\mu_1, \mu_2 \in \Probability{\Omega_*}$.

This is a linear inverse problem on a probability space with a weighted regularization term added to its target functional. Again, linear inverse problems on measure spaces have been extensively studied in the past, see \eg \cite{casas2012approximation,casas2014optimal,casas2019optimal,clason2021entropic,pieper2013priori}, just to mention a few. Still, \textcolor{cyan}{to the author's best knowledge}, the above optimization \textcolor{cyan}{problem} is unique \textcolor{cyan}{in that} it considers a regularization \wrt the Wasserstein distance instead of the Radon norm. We expect this regularization to be advantageous, because the Wasserstein distance metrizes the weak convergence of probability measures (on Polish spaces), see \eg \cite[Chapter 6]{villani2009optimal}.

Even though the problem proposed in \cref{pr:INTRO-WasserInvProb} is a convex optimization problem, it is a nontrivial task to actually solve such kind of problems. This is not only due to the potential complexity of evaluating $G$, but also to the fact that evaluating the $\rho$-Wasserstein distance involves solving a nonlinear optimization problem on $\Probability{\Omega_* \times \Omega_*}$. The latter task is subject to the ``curse of dimensionality'', meaning that the computational complexity \textcolor{cyan}{is likely to increase drastically as higher approximation quality is pursued}. In \cite{hillbrecht2022bilevel1, hillbrecht2022bilevel2}, we took advantage of a quadratic Tikhonov regularization to improve the computational properties of the Kantorovich problem \cref{pr:INTRO-Kant} and, \textcolor{cyan}{in turn}, by both replacing the lower-level Kantorovich problem by its $L^2$ regularized counterpart and interposing a smoothing of the involved variables, also the properties of the general bilevel problem \cref{pr:INTRO-BilevelKant}. We have shown that both the non-regularized and the regularized bilevel problems are well-posed and that, under some mild regularity conditions on the domains and the given data, by driving the regularization and smoothing parameters to zero, we can approximate certain solutions to the general bilevel problem by means of solutions to the regularized bilevel problems. This approach has been demonstrated to reduce computational complexity by circumventing the "curse of dimensionality." However, this reduction in complexity often comes at the expense of structural properties of the bilevel optimization problem in question. In particular, the regularized version of the Wasserstein inverse problem \cref{pr:INTRO-WasserInvProb} will most likely no longer be a convex problem, see also the discussion in \cite[Section 4.2]{hillbrecht2022bilevel2} for the case where $G$ is the solution operator of an elliptic PDE.

In the same spirit, the present paper deals with a discretized version of the Kantorovich problem, which is also known as the Hitchcock problem (\textcolor{cyan}{or classical transportation problem}), see problem \cref{pr:PROB-Hitch} below. Similar to the infinite-dimensional case of the Kantorovich problem, we concern ourselves with a bilevel problem that has the Hitchcock problem as one of its constraints, see problem \cref{pr:PROB-BilevelHitch} below. Bilevel optimization problems with linear lower-level structure have been studied extensively in the literature (see, \eg \cite{mehlitz2021note} and the references therein). In such "fully linear" settings, concepts like partial calmness can be employed to establish equivalences to penalized single-level problems. While these approaches are powerful, the focus of
this paper lies elsewhere: we are interested in the interaction between the geometry of transport constraints and solution existence, especially in cases where the Hitchcock problem may admit non-unique solutions and where standard reformulations fail to capture this complexity. Exploring connections with penalty-based methods
remains an interesting direction for future work.

In the present paper, we introduce a quadratic regularization to improve the numerical properties of both the Hitchcock problem and the corresponding general bilevel problem. However, in contrast to the case of the Kantorovich problem and owing to the finite-dimensional character of the Hitchcock problem, there is no additional smoothing of the variables involved. We replicate the results of the previous papers and in particular improve the results concerning the existence of recovery sequences (which ensure the upper-bound condition in the framework of $\Gamma$-convergence). Furthermore, we introduce an additional regularization of the dual problem of the regularized Hitchcock problem, which allows us to compute (sub-)derivatives of the regularized marginal-to-transport plan mapping. In the case that the bilevel objective functional is sufficiently smooth, we then may employ a nonsmooth optimization algorithm to \textcolor{cyan}{(heuristically)} solve the bilevel problem. In some sense, the present paper is meant to be the basis for the numerical treatment of (discretized) bilevel Kantorovich problems such as the transportation identification problem or the Wasserstein inverse problem \cref{pr:INTRO-WasserInvProb} \textcolor{cyan}{defined} above.

The present paper is structured as follows: In \cref{sc:PROB}, we derive the so-called Hitchcock problem as a special case from the infinite-dimensional Kantorovich problem and address the existence of solutions to both itself as well as the corresponding general bilevel problem. We then introduce a quadratic regularization to the Hitchcock problem's target function, which we, in turn, use to regularize the bilevel problem. Furthermore, we show that (in the presence of a recovery sequence), for a sequence of vanishing regularization parameters, any cluster point of the sequence of regularized bilevel solutions is a solution to the general bilevel problem. We then conclude the section with an explicit construction of such a recovery sequence. \cref{sc:DERIV} introduces a further quadratic regularization of the dual problem of the regularized Hitchcock problem. Thereby, we can define a regularized marginal-to-transport-plan mapping, whose differentiability properties \textcolor{cyan}{are investigated} in detail and which forms the basis of the numerical treatment of the regularized bilevel problems. Finally in \cref{sc:NUM}, we present the results of some first numerical tests to verify our findings from the previous sections.

As a disclaimer, most of the results of the present paper can be found, reasonably more detailed, in Part II of the doctoral thesis \cite{hillbrecht2024quadratic}. The present work nevertheless provides added value in that we shorten some of the arguments and provide a different perspective on some of the results.

% % % % % % % % % % % % % % % % % % % % % % % % % % % % % % % % % % % % % % %

\section{The Discrete Kantorovich Problem aka Hitchcock's Problem}\label{sc:PROB}
Given $n_1, n_2 \in \N$, we consider the finite sets $\Omega_1 = \{1, \dots, n_1\}$ and $\Omega_2 = \{1, \dots, n_2\}$, which we both endow with the discrete topology. We abbreviate their Cartesian product by $\Omega \coloneqq \Omega_1 \times \Omega_2$. Carrying the discrete topology, the Borel sigma algebras $\Borel{\Omega_1}$, $\Borel{\Omega_2}$, and $\Borel{\Omega}$ are simply the power sets $\Pcal(\Omega_1)$, $\Pcal(\Omega_2)$, and $\Pcal(\Omega)$, respectively.

The above setting implies the isometric isomorphisms $\Measure{\Omega_i} \cong \R^{n_i}$, $i = 1, 2$, as well as $\Measure{\Omega}, C(\Omega) \cong \R^{n_1 \times n_2}$. Consequently, any regular Borel measure $\mu_i \in \Measure{\Omega_i}$, $i = 1, 2$, can be represented by some vector $\vec{\mu}_i \in \R^{n_i}$ \textcolor{cyan}{which} satisfies $\norm[\Measure{\Omega_i}]{\mu_i} = \norm[1]{\vec{\mu}_i}$. Of course, the same holds \textcolor{cyan}{for} the elements of $\Measure{\Omega}$ and $C(\Omega)$, \ie for every measure $\theta \in \Measure{\Omega}$ and every function in $f \in C(\Omega)$, we can always find matrix representations $M, F \in \R^{n_1 \times n_2}$ such that $\norm[\Measure{\Omega}]{\theta} = \norm[1]{M} = \sum_{(i_1, i_2) \in \Omega} |M_{i_1, i_2}|$ and $\norm[C(\Omega)]{f} = \norm[\infty]{F} = \max_{(i_1, i_2) \in \Omega} |F_{i_1, i_2}|$. To ease the notation, in the following we will use the same symbol for elements of the measure spaces / function spaces and their representing vectors and matrices.

With this convention, we find that the Kantorovich problem \cref{pr:INTRO-Kant} from \cref{sc:INTRO} is equivalent to the problem
\begin{equation}
\label{pr:PROB-Hitch}
\tag{H}
	\begin{array}{rl}
		\inf\limits_{\pi} & \scalarproduct[F]{c}{\pi} \\[0.5em]
		\text{s.t.} & \pi \in \R^{n_1 \times n_2}, \quad \pi \geq 0, \\[0.25em]
		& \pi \one = \mu_1, \quad \pi^\top \one = \mu_2.
	\end{array}
\end{equation}
This standard linear problem is known as the \emph{Hitchcock problem} (\textcolor{cyan}{or classical transportation problem}) and finds widespread application throughout different fields of mathematics and economics, see \eg \cite{ford2015flows} and the references therein. In the above and \textcolor{cyan}{in} the rest of this paper, the symbol $\one$ refers to the vector (of any dimension) that consists only of ones, \ie $\one = (1, \dots, 1)^\top$. Moreover, $\scalarproduct[F]{\substitute}{\substitute}$ denotes the Frobenius scalar product.

The Hitchcock problem is the problem of minimizing a continuous target function over a compact feasible set. We therefore immediately receive the following result:

\begin{theorem}
\label{th:PROB-HitchSolu}
	For any pair of compatible marginals $\mu_1 \in \R^{n_1}$ and $\mu_2 \in \R^{n_2}$ with $\mu_1, \mu_2 \geq 0$ and $\mu_1^\top \one = \mu_2^\top \one$ as well as for any cost matrix $c \in \R^{n_1 \times n_2}$, the Hitchcock problem \cref{pr:PROB-Hitch} has at least one optimal solution.
\end{theorem}

Analogously to the case of the infinite-dimensional bilevel Kantorovich problem, which we introduced in \cref{sc:INTRO} and discussed in detail in \textcolor{cyan}{the first two papers} \cite{hillbrecht2022bilevel1,hillbrecht2022bilevel2}, in this paper we are interested in a bilevel problem that is governed by the finite-dimensional Hitchcock problem.

Let $\Jcal \colon \R^{n_1 \times n_2} \times \R^{n_1} \to \R$ be some given lower semicontinuous objective function that is bounded on bounded sets, \ie for all $M > 0$ it holds that
\begin{equation*}
	\sup_{\norm{(\pi, \mu_1)} < M} \Jcal(\pi, \mu_1) < \infty.
\end{equation*}
Furthermore, consider some fixed target marginal $\mu_2^\drm \in \R^{n_2}$, satisfying $\mu_2^\drm \geq 0$, \textcolor{cyan}{as well as} some fixed cost matrix $c_\drm \in \R^{n_1 \times n_2}$. For the \textcolor{cyan}{rest} of this paper, we are \textcolor{cyan}{mainly} interested in the \emph{bilevel Hitchcock problem} that is given by
\begin{equation}
\tag{BH}
\label{pr:PROB-BilevelHitch}
	\begin{array}{rl}
		\displaystyle \inf_{\pi, \mu_1} & \Jcal(\pi, \mu_1) \\[0.5em]
		\text{s.t.} & \pi \in \R^{n_1 \times n_2}, \quad \mu_1 \in \R^{n_1}, \\[0.25em]
		& \mu_1 \geq 0, \quad \mu_1^\top \one = {\mu_2^\drm}^\top \one, \\[0.5em]
		& \pi ~ \text{solves \cref{pr:PROB-Hitch} w.r.t.} ~ \mu_1, ~ \mu_2^\drm, ~ \text{and} ~ c_\drm.
	\end{array}
\end{equation}
The bilevel Hitchcock problem \cref{pr:PROB-BilevelHitch} can be seen as a discretization of the bilevel Kantorovich problem \cref{pr:INTRO-BilevelKant}. It will therefore form the basis for the numerical treatment of problems such as the transportation identification problem as well as the Wasserstein inverse problem \cref{pr:INTRO-WasserInvProb}, which \textcolor{cyan}{were} briefly motivated in \cref{sc:INTRO} and thoroughly discussed in \cite[Section 4]{hillbrecht2022bilevel2}.

\begin{lemma}
\label{lm:PROB-FeasSetComp}
	For any $\mu_2 \in \R^{n_2}$ and any $c \in \R^{n_1 \times n_2}$, the set
	\begin{equation*}
		\Fcal
		\coloneqq
		\bigl\{
			(\pi, \mu_1) %\in \R^{n_1 \times n_2} \times \R^{n_1}
			\colon \mu_1 \geq 0, ~ \mu_1^\top \one = \mu_2^\top \one, ~ \pi \, \text{solves \cref{pr:PROB-Hitch} w.r.t.} \, \mu_1, \, \mu_2, \, c
		\bigr\}
	\end{equation*}
	is non-empty and compact.
\end{lemma}

\begin{proof}
	The non-emptiness of $\Fcal$ is guaranteed by \cref{th:PROB-HitchSolu}. Its boundedness follows from the linear constraints of \cref{pr:PROB-Hitch}. To see that it is also closed, let us consider a sequence $(\pi_k, \mu_{1,k})_{k \in \N} \subset \Fcal$ with $(\pi_k, \mu_{1,k}) \to (\pi, \mu_1)$ as $k \to \infty$.
%	{\color{cyan}
	The optimality system of \cref{pr:PROB-Hitch} \wrt $\mu_{1, k}$, $\mu_2$, and $c$ is given by
	\begin{align}
	\label{eq:PROB-HitchOptSys-a} \pi_k \one = \mu_{1,k}, \quad \pi_k^\top \one = \mu_2, \quad \pi_k &\geq 0, \\
	\label{eq:PROB-HitchOptSys-b} c- \alpha_{1,k} \oplus \alpha_{2,k} &\geq 0, \\
	\label{eq:PROB-HitchOptSys-c} (c - \alpha_{1,k} \oplus \alpha_{2,k}, \pi_k)_F &= 0,
	\end{align}
	with $\alpha_{1,k} \in \R^{n_1}$ and $\alpha_{2,k} \in \R^{n_2}$, see \cite[Theorem 2]{flood1953hitchcock}. Owing to the convergence $\pi_k \to \pi$ as $k \to \infty$, if $\pi_{i, j} > 0$, then also $(\pi_k)_{i, j} > 0$ for all $k$ large enough. Thus, if $\alpha_{1,k}$ and $\alpha_{2,k}$ satisfy \cref{eq:PROB-HitchOptSys-b} and \cref{eq:PROB-HitchOptSys-c} \wrt $c$ and $\pi_k$, they satisfy the same equalities \wrt $c$ and $\pi$. This together with the feasibility of the limit $\pi$ for \cref{pr:PROB-Hitch} \wrt $\mu_1$ and $\mu_2$, which can be obtained by passing to the limit in \cref{eq:PROB-HitchOptSys-a}, shows that, for all $k$ large enough, $\alpha_{1,k}$ and $\alpha_{2,k}$ solve the optimality system for $\pi$. Since the optimality system is sufficient for optimality, this proves the claim.
%	}
\end{proof}

\cref{lm:PROB-FeasSetComp} in conjunction with the presupposed lower semicontinuity of the bilevel objective immediately yields the existence of an optimal solution for the bilevel Hitchcock problem:

\begin{theorem}
\label{th:PROB-BilevelHitchSolu}
	For any lower semicontinuous objective function $\Jcal$ and any given data $\mu_2^\drm$ and $c_\drm$, the bilevel Hitchcock problem \cref{pr:PROB-BilevelHitch} has at least one optimal solution.
\end{theorem}

\begin{remark}
	Note that the boundedness of $\Jcal$ on bounded sets does not play a role in the proof of the existence result in \cref{th:PROB-BilevelHitchSolu}. We need it \textcolor{cyan}{later}, however, to guarantee the boundedness of the sequence of regularized bilevel solutions in order to be able to extract a cluster point that solves \cref{pr:PROB-BilevelHitch}. For this reason, we have assumed the boundedness of the bilevel objective from the very beginning.
\end{remark}

Even though the subordinate problem in \cref{pr:PROB-BilevelHitch} is just an finite-dimensional \textcolor{cyan}{linear problem}, which can be easily solved by a variety of efficient solvers, we are facing the same \textcolor{cyan}{difficulties} as was the case in the infinite-dimensional setting of the bilevel Kantorovich problem in \cite{hillbrecht2022bilevel1}:
\begin{itemize}
	\item the solution to the Hitchcock problem, living on the Cartesian product of the marginal's domains, is a high-dimensional object and subject to the "curse of dimensionality"; therefore, in general, the bilevel Hitchcock problem turns out to be a high-dimensional optimization problem.
	
	\item there is no closed-form expression for the solution of the Hitchcock problem (unless the cost matrix is constant or at least one of the marginals is a scaled unit vector); this prevents us from eliminating the optimal transport plan from the set of optimization variables, again resulting in an high-dimensional optimization problem.
	
	\item intensifying the previous \textcolor{cyan}{difficulty}, the optimal transport plan does not even need to be unique; therefore, there may not even be a single-valued solution map, ruling out the applicability of the popular implicit programming approach (\textcolor{cyan}{as seen \eg in} \cite{kovcvara1997topology} or in \cite{hintermueller2016bundle-free}).
\end{itemize}

The standard strategy to tackle these difficulties, \textcolor{cyan}{which arise} from the Hitchcock problem's \textcolor{cyan}{inherent structure}, is the so-called entropic regularization. \textcolor{cyan}{The latter} introduces a logarithmic regularization term to the Hitchcock problem's target function. This results not only in a unique optimal transport plan but also drastically reduces the dimension of the optimization problem by means of its dual formulation. Moreover, the entropic regularization allows for ``lightspeed'' computation of optimal transport plans via the famous Sinkhorn algorithm. For details on the entropic regularization and its application to optimal transport we only refer to the groundbreaking paper \cite{cuturi2013sinkhorn}.

% % % % % % % % % % % % % % % % % % % % % % % % % % % % % % % % % % % % % % %

\subsection{Tikhonov Regularization of the Hitchcock Problem}\label{sc:REG}
Instead of employing entropic regularization to overcome the aforementioned challenges, we use the following regularization approach for the Kantorovich problem, which was first introduced in \cite{lorenz2021quadratically}: for some regularization parameter $\gamma > 0$, we consider the \emph{regularized Hitchcock problem}
\begin{equation}
\label{pr:REG-RegHitch}
\tag{H$_\gamma$}
	\begin{array}{rl}
		\inf\limits_{\pi} & \scalarproduct[F]{c}{\pi} + \frac\gamma2 \norm[F]{\pi}^2 \\[0.5em]
		\text{s.t.} & \pi \in \R^{n_1 \times n_2}, \quad \pi \geq 0, \\[0.25em]
		& \pi \one = \mu_1, \quad \pi^\top \one = \mu_2,
	\end{array}
\end{equation}
which is almost identical to \cref{pr:PROB-Hitch}, with the only difference being that its target function additionally accommodates a quadratic regularization term including the Frobenius norm, therefore resulting in an strictly convex continuous objective. The regularization does not affect the compactness of its feasible set. \textcolor{cyan}{Therefore}, just like in the previous section, we \textcolor{cyan}{immediately receive} the following result.

\begin{theorem}
\label{th:REG-RegHitchSolu}
	For any pair of compatible marginals $\mu_1 \in \R^{n_1}$ and $\mu_2 \in \R^{n_2}$, with $\mu_1, \mu_2 \geq 0$ \textcolor{cyan}{as well as} $\mu_1^\top \one = \mu_2^\top \one$, \textcolor{cyan}{and} for any cost matrix $c \in \R^{n_1 \times n_2}$, the regularized Hitchcock problem \cref{pr:REG-RegHitch} has a unique optimal solution $\pi_\gamma \in \R^{n_1 \times n_2}$.
\end{theorem}

In the following lemma, we characterize \cref{pr:REG-RegHitch}'s unique solution by means of \cref{pr:REG-RegHitch}'s necessary and sufficient first-order conditions.

\begin{theorem}
\label{th:REG-RegHitchDualSys}
	$\pi_\gamma \in \R^{n_1 \times n_2}$ is the unique optimal solution to \cref{pr:REG-RegHitch} (\wrt $\mu_1 \in \R^{n_1}$, $\mu_2 \in \R^{n_2}$, and $c \in \R^{n_1 \times n_2}$) if and only if there exist $\alpha_1 \in \R^{n_1}$ and $\alpha_2 \in \R^{n_2}$ such that
	\begin{equation}
	\label{eq:REG-DualSystem}
		\pi_\gamma = \frac1\gamma (\alphasum)_+, \quad
		\pi_\gamma \one = \mu_1, \quad \pi_\gamma^\top \one = \mu_2.
	\end{equation}
	where $(v_1 \oplus v_2)_{i_1, i_2} = v_1^{i_1} + v_2^{i_2}$, for all $(i_1, i_2) \in \Omega$, denotes the outer sum of the vectors $v_1 \in \R^{n_1}$ and $v_2 \in \R^{n_2}$
\end{theorem}

\begin{proof}
	Because \cref{pr:REG-RegHitch}'s target function is convex and differentiable and \textcolor{cyan}{because} the constraints are (affine) linear, $\pi_\gamma$ is a solution to \cref{pr:REG-RegHitch} if and only if there exist $\alpha_i \in \R^{n_i}$, $i = 1,2$, and $\lambda \geq 0$ such that
	\begin{align*}
%		\begin{split}
			\pi_\gamma \one = \mu_1, \quad \pi_\gamma^\top \one = \mu_2, \quad \pi_\gamma &\geq 0, \\
			c + \gamma \pi_\gamma - \Sigma_1^* \alpha_1 - \Sigma_2^* \alpha_2 - \lambda &= 0, \\
			\scalarproduct[F]{\lambda}{\pi_\gamma} &= 0,
%		\end{split}
	\end{align*}
	where $\Sigma_1 \colon \R^{n_1 \times n_2} \to \R^{n_1}$ and $\Sigma_2 \colon \R^{n_1 \times n_2} \to \R^{n_2}$ refer to the mappings $\theta \mapsto \theta \one$ and $\theta \mapsto \theta^\top \one$, respectively. Their adjoint mappings \textcolor{cyan}{$\Sigma_1^*$ and $\Sigma_2^*$} relate to the outer sum operator via the equality
	\begin{equation*}
		\Sigma_1^* \alpha_1 + \Sigma_2^* \alpha_2
		= \alpha_1 \oplus \alpha_2.
	\end{equation*}
	The above first-order system is therefore equivalent to
	\begin{align*}
%	\label{eq:REG-FirstOrderSystem}
%		\begin{split}
			\pi_\gamma \one = \mu_1, \quad \pi_\gamma^\top \one = \mu_2, \quad \pi_\gamma &\geq 0, \\
			c + \gamma \pi_\gamma - \alpha_1 \oplus \alpha_2 &\geq 0, \\
			\scalarproduct[F]{c + \gamma \pi_\gamma - \alpha_1 \oplus \alpha_2}{\pi_\gamma} &= 0,
%		\end{split}
	\end{align*}
	which is in turn equivalent to \cref{eq:REG-DualSystem}.
\end{proof}

\cref{th:REG-RegHitchDualSys} shows why, in our opinion, a quadratic regularization of the Hitchcock problem is superior to \textcolor{cyan}{an} entropic regularization. \textcolor{cyan}{It offers} the same advantageous properties as \textcolor{cyan}{an} entropic regularization, \ie
\begin{itemize}
	\item uniqueness of the solution;
	
	\item representation of the solution by means of dual variables, implying a reduction of dimensions;
	
	\item efficient computation of the solution by means of standard methods; \textcolor{cyan}{for the former, we refer to} \cite[Section 3]{lorenz2021quadratically}; \textcolor{cyan}{for the latter, we refer to \cite{cuturi2013sinkhorn}}.
\end{itemize}
\textcolor{cyan}{Besides that,} thanks to \textcolor{cyan}{presence of the} $(\substitute)_+$-operator in \cref{eq:REG-DualSystem}, \textcolor{cyan}{it also promotes} the sparsity of the optimal transport plan, which is a unique feature of the solution of the non-regularized Hitchcock problem, \textcolor{cyan}{which is entirely ignored} by \textcolor{cyan}{an} entropic regularization \textcolor{cyan}{of the Hitchcock problem}. In \cref{sc:DERIV} we will see that the presence of the $(\substitute)_+$-operator in the first-order optimality system of the Kantorovich problem will also prove useful when calculating derivatives.

\textcolor{cyan}{Analogously} to \cite{lorenz2021quadratically}, we receive the following equivalent formulation of the dual problem of the regularized Hitchcock problem.

\begin{lemma}
	The Lagrangian dual problem to \cref{pr:REG-RegHitch} is equivalent to the problem
	\begin{equation}
	\label{pr:REG-DualRegHitch}
	\tag{D$_\gamma$}
		\begin{array}{rl}
			\sup\limits_{\alpha_1, \alpha_2} & \scalarproduct{\alpha_1}{\mu_1} + \scalarproduct{\alpha_2}{\mu_2} - \frac{1}{2\gamma} \norm[F]{(\alphasum)_+}^2 \\
			\text{s.t.} & \alpha_i \in \R^{n_i}, ~ i = 1, 2.
		\end{array}
	\end{equation}
	For each $\gamma > 0$, there is an optimal solution to \cref{pr:REG-DualRegHitch} and there is no duality gap, \ie $\inf \cref{pr:REG-RegHitch} = \sup \cref{pr:REG-DualRegHitch}$.
\end{lemma}

\cref{th:REG-RegHitchSolu} ensures the uniqueness of the solution to the regularized Hitchcock problem and \cref{th:REG-RegHitchDualSys} \textcolor{cyan}{provides} an algorithmic advantage, since the dual representation of the solution leads to a reduction of dimensions. This \textcolor{cyan}{yields a significant} advantage over the non-regularized Hitchcock problem. We consequently replace the lower-level Hitchcock problem from \cref{pr:PROB-BilevelHitch} by the Tikhonov regularized Hitchcock problem \cref{pr:REG-RegHitch}, \textcolor{cyan}{where} $\gamma > 0$, to receive the \emph{regularized bilevel Hitchcock problem}
\begin{equation}
\label{pr:REG-RegBilevelHitch}
\tag{BH$_\gamma$}
	\begin{array}{rl}
		\inf\limits_{\pi, \mu_1, c} & \Jcal(\pi, \mu_1) + \textcolor{cyan}{\frac1{2\gamma} \norm[F]{c - c_\drm}^2 }  \\[0.5em]
		\text{s.t.} & \pi, c \in \R^{n_1 \times n_2}, \quad \mu_1 \in \R^{n_1}, \\[0.25em]
		& \mu_1 \geq 0, \quad \mu_1^\top \one = {\mu_2^\drm}^\top \one, \\[0.5em]
		& \pi ~ \text{(uniquely) solves \cref{pr:REG-RegHitch} w.r.t.} ~ \mu_1, ~ \mu_2^\drm, ~ \text{and} ~ c,
	\end{array}
\end{equation}
\textcolor{cyan}{in which} $\Jcal$, $\mu_2^\drm$, and $c_\drm$ are \textcolor{cyan}{given as} in the formulation of the non-regularized bilevel problem \cref{pr:PROB-BilevelHitch}.

Note that in comparison to \cref{pr:PROB-BilevelHitch}, in the above bilevel problem, we have not only replaced the lower-level problem by its regularized counterpart but also introduced the cost function as an optimization variable as well. The motivation for this modification becomes evident in \textcolor{cyan}{the proof of \cref{co:Approx-ExistRecovSeq}}, where we present a construction of a recovery sequence in which the $\gamma$-dependent parts of the sequence are hidden in the cost function.

The existence of optimal solutions to the regularized bilevel Hitchcock problem is completely along the lines of the corresponding result of the non-regularized bilevel Hitchcock problem, see \cref{th:PROB-BilevelHitchSolu}. We therefore have the following result:

\begin{theorem}
\label{th:REG-RegBilevelHitchSolu}
	For any lower semicontinuous objective function $\Jcal$ and any given data $\mu_2^\drm$ and $c_\drm$, the regularized bilevel Hitchcock problem \cref{pr:REG-RegBilevelHitch} has at least one optimal solution.
\end{theorem}

% % % % % % % % % % % % % % % % % % % % % % % % % % % % % % % % % % % % % % %

\subsection{Approximation of Bilevel Solutions}\label{sc:APPROX}
In general, it is not clear how the solutions to \cref{pr:REG-RegBilevelHitch} relate to the solutions to \cref{pr:PROB-BilevelHitch}. In particular, \textcolor{cyan}{it is not clear} whether the solutions to the latter can be approximated by solutions to the former\textcolor{cyan}{,} if the regularization parameter is driven towards $0$. In the rest of this section, \textcolor{cyan}{we will show that a positive answer can be given.}

To this end, consider a sequence of regularization parameters $(\gamma_k)_{k \in \N} \subset \R_{>0}$, with $\gamma_k \to 0$ as $k \to \infty$. \textcolor{cyan}{Moreover}, consider the sequence of solutions $(\pi_k, \mu_{1,k}, c_k)_{k \in \N}$ to the sequence of regularized bilevel Hitchcock problems \hyperref[pr:REG-RegBilevelHitch]{(BH$_{\gamma_k}$)$_{k \in \N}$}.

\begin{lemma}
\label{lm:APPROX-ClustPointFeas}
	The sequence of regularized solutions $(\pi_k, \mu_{1,k}, c_k)_{k \in \N}$
%	to the sequence of regularized bilevel Hitchcock problems \textnormal{(\hyperref[pr:REG-RegBilevelHitch]{BH$_{\gamma_k}$})}$_{k \in \N}$
	has a cluster point $(\bar{\pi}, \bar{\mu}_1, c_\drm)$ and the point $(\bar{\pi}, \bar{\mu}_1)$ is feasible for \cref{pr:PROB-BilevelHitch}, \ie $\bar{\mu}_1 \geq 0$ as well as $\bar{\mu}_1^\top \one = {\mu_2^\drm}^\top \one$ and $\bar{\pi}$ solves \cref{pr:PROB-Hitch} \wrt $\bar{\mu}_1$, $\mu_2^\drm$, and $c_\drm$.
\end{lemma}

\begin{proof}
	The constraints of \cref{pr:PROB-Hitch} imply the boundedness of $(\pi_k, \mu_{1,k})_{k \in \N}$ and therefore the existence of a cluster point $(\bar{\pi}, \bar{\mu}_1)$ such that (after possibly extracting a subsequence) $(\pi_k, \mu_{1,k}) \to (\bar{\pi}, \bar{\mu}_1)$ as $k \to \infty$. For any $k \in \N$, by \cref{th:REG-RegHitchSolu}, there exists an optimal solution to \hyperref[pr:REG-RegHitch]{(H$_{\gamma_k}$)} \wrt $\bar{\mu}_1$, $\mu_2^\drm$, and $c_\drm$, which we denote by $\tilde{\pi}_k$. Combining this with the optimality of $(\pi_k, \mu_{1,k}, c_k)$ for \hyperref[pr:REG-RegBilevelHitch]{(BH$_{\gamma_k}$)}, we find that
	\begin{equation*}
		\norm[F]{c_k - c_\drm}
		\leq \textcolor{cyan}{(2 \gamma_k)^{\frac12}} \bigl( \Jcal(\tilde{\pi}_k, \bar{\mu}_1) - \Jcal(\pi_k, \mu_{1,k}) \bigr).
	\end{equation*}
	Because \textcolor{cyan}{$\Jcal$ is bounded} on bounded sets and because \textcolor{cyan}{the regularization parameter} $\gamma_k$ vanishes, \textcolor{cyan}{we receive} the convergence $c_k \to c_\drm$ as $k \to \infty$. \textcolor{cyan}{This} establishes the first claim of the lemma.
	
	To prove the second claim, we first note that the conditions on $\bar{\mu}_1$ follow from the linearity of the constraints in \hyperref[pr:REG-RegBilevelHitch]{(BH$_{\gamma_k}$)}. Similarly, the linearity of the constraints of the regularized Hitchcock problems yields that $\bar{\pi}$ is feasible for \cref{pr:REG-RegHitch} \wrt $\bar{\mu}_1$ and $\mu_2^\drm$. To show that it is also optimal \wrt $c_\drm$, just like in \cref{lm:PROB-FeasSetComp}, we can find an optimal solution $\pi^*$ and a sequence $(\pi_k^*)_{k \in \N}$ such that $\pi_k^* \in \Coupling{\mu_{1,k}}{\mu_2^\drm}$ for all $k \in \N$ and $\pi_k^* \to \pi^*$ as $k \in \N$. Hence,
	\begin{equation*}
		\scalarproduct[F]{c_\drm}{\bar{\pi}}
		\leq \lim_{k \to \infty} \scalarproduct[F]{c_k}{\pi_k} + \frac{\gamma_k}2 \norm[F]{\pi_k}^2
		\leq \lim_{k \to \infty} \scalarproduct[F]{c_k}{\pi_k^*} + \frac{\gamma_k}2 \norm[F]{\pi_k^*}^2
		= \scalarproduct[F]{c_\drm}{\pi^*},
	\end{equation*}
	which proves the claim.
\end{proof}

We now arrive at an important result which guarantees the approximability of solutions to the non-regularized bilevel problems by solutions to the regularized bilevel problems, \textcolor{cyan}{provided} a so-called recovery sequence \textcolor{cyan}{is given}.

\begin{theorem}
\label{th:APPROX-ClustPointOptimal}
	Let $(\pi^*, \mu_1^*)$ be an optimal solution to the non-regularized bilevel Hitchcock problem \cref{pr:PROB-BilevelHitch} that is accompanied by a recovery sequence, \textcolor{cyan}{which is a} sequence $(\pi_k^*, \mu_{1,k}^*, c_k^*)_{k \in \N}$ \textcolor{cyan}{satisfying}
	\begin{itemize}
		\item[i.] \textcolor{cyan}{for all $k \in \N$, the point} $(\pi_k^*, \mu_{1,k}^*, c_k^*)$ is feasible for \textnormal{\hyperref[pr:REG-RegBilevelHitch]{(BH$_{\gamma_k}$)}},
		
		\item[ii.] $\limsup_{k \to \infty} \Jcal(\pi_k^*, \mu_{1,k}^*) + \textcolor{cyan}{\frac1{2 \gamma_k} \norm[F]{c_k^* - c_\drm}^2} \leq \Jcal(\pi^*, \mu_1^*)$.
	\end{itemize}
	Then, the point $(\bar{\pi}, \bar{\mu}_1)$ \textcolor{cyan}{from \cref{lm:APPROX-ClustPointFeas}} is \textcolor{cyan}{an optimal solution to} the non-regularized bilevel Hitchcock problem \cref{pr:PROB-BilevelHitch}.
\end{theorem}

\begin{proof}
	Up to subsequences, $(\pi_k, \mu_{1,k}) \to (\bar{\pi}, \bar{\mu}_1)$ as $k \to \infty$. Because $\Jcal$ was assumed to be lower semicontinuous and because $(\pi_k, \mu_{1,k}, c_k)$ is optimal for \hyperref[pr:REG-RegBilevelHitch]{(BH$_{\gamma_k}$)},
	\begin{align*}
		\Jcal(\bar{\pi}, \bar{\mu}_1)
		&\leq \liminf_{k \to \infty} \, \Jcal(\pi_k, \mu_{1,k}) + \textcolor{cyan}{\frac1{2 \gamma_k} \norm[F]{c_k - c_\drm}^2} \\
		&\leq \limsup_{k \to \infty} \, \Jcal(\pi_k^*, \mu_{1,k}^*) + \textcolor{cyan}{\frac1{2 \gamma_k} \norm[F]{c_k^* - c_\drm}^2}
		\leq \Jcal(\pi^*, \mu_1^*),
	\end{align*}
	which proves the claim, since, by \cref{lm:APPROX-ClustPointFeas}, $(\bar{\pi}, \bar{\mu}_1)$ is feasible and $(\pi^*, \mu_1^*)$ was assumed to be optimal for \cref{pr:PROB-BilevelHitch}.
\end{proof}

\begin{remark}
	\begin{enumerate}
		\item \textcolor{cyan}{In variational convergence theory, particularly in the context of $\Gamma$-convergence of functionals,} the assumption of the existence of a recovery sequence is standard, \textcolor{cyan}{since} it ensures that the $\limsup$ inequality holds. In the absence of such a sequence, convergence of minimizers of regularized problems to minimizers of non-regularized problems \textcolor{cyan}{may} fail, as the following example illustrates:
			
		Consider the Hitchcock problem \cref{pr:PROB-Hitch} \textcolor{cyan}{\wrt} the marginals $\mu_1^\drm = (1, 1, 0)^\top$ \textcolor{cyan}{as well as} $\mu_2^\drm = (0, 1, 1)^\top$ \textcolor{cyan}{and \wrt to} the metric cost matrix $c_\drm \in \R^{3 \times 3}$ given by $c_{i_1, i_2} = |i_1 - i_2|$ for $i_1, i_2 = 1, \dots, 3$. It is straightforward to show that
		\begin{equation*}
			\pi_1^*
			= \textcolor{cyan}{\begin{pmatrix}
				0 & 1 & 0 \\
				0 & 0 & 1 \\
				0 & 0 & 0
			\end{pmatrix}},
			\quad
			\pi_2^*
			= \textcolor{cyan}{\begin{pmatrix}
				0 & 0 & 1 \\
				0 & 1 & 0 \\
				0 & 0 & 0
			\end{pmatrix}},
			\quad \text{and} \quad
			\pi_3^*
			= \begin{pmatrix}
				0 & \frac12 & \frac12 \\
				0 & \frac12 & \frac12 \\
				0 & 0 & 0
			\end{pmatrix}
		\end{equation*}
		each are optimal solutions to \cref{pr:PROB-Hitch}. Moreover, \textcolor{cyan}{considering} the lower semicontinuous and bounded objective function
		\begin{equation*}
			\Jcal(\pi, \mu_1)
			= \begin{cases}
				0, & \text{if} ~ \pi = \pi_1^*, \\
				1, & \text{else},
			\end{cases}
			\quad + \quad
			\begin{cases}
				0, & \text{if} ~ \mu_1 = \mu_1^\drm, \\
				1, & \text{else},
			\end{cases}
		\end{equation*}
		one finds that $(\pi_1^*, \mu_1^\drm)$ is the unique solution to \cref{pr:PROB-BilevelHitch} \wrt $\Jcal$, $\mu_2^\drm$, and $c_\drm$. One can then try to approximate the solution to \cref{pr:PROB-BilevelHitch} by solutions of the regularized bilevel problem
		\begin{equation}
		\label{pr:APPROX-BilevelHitchNoCostReg}
		\tag{$\widetilde{\text{BH}}_\gamma$}
		\begin{array}{rl}
			\inf\limits_{\pi, \mu_1} & \Jcal(\pi, \mu_1) \\[0.5em]
			\text{s.t.} & \pi \in \R^{n_1 \times n_2}, \quad \mu_1 \in \R^{n_1}, \\[0.25em]
			& \mu_1 \geq 0, \quad \mu_1^\top \one = {\mu_2^\drm}^\top \one, \\[0.6em]
			& \pi ~ \text{solves \cref{pr:REG-RegHitch} w.r.t.} ~ \mu_1, ~ \mu_2^\drm, ~ \text{and} ~ c_\drm.
		\end{array}
		\end{equation}
		Note that this problem coincides with \cref{pr:REG-RegBilevelHitch} if, in the latter, the cost is removed from the set of optimization variables and the penalty term is dropped from the objective function---modifications introduced solely to facilitate the construction of a recovery sequence in \textcolor{cyan}{\cref{co:Approx-ExistRecovSeq}}.
		
		However, for each $\gamma > 0$, the unique solution to \cref{pr:REG-RegHitch} \wrt $\mu_1^\drm$, $\mu_2^\drm$, and $c_\drm$ is given by $\pi_3^*$. Consequently, for any feasible point $(\pi, \mu_1)$ for \cref{pr:APPROX-BilevelHitchNoCostReg}, it holds that $\Jcal(\pi, \mu_1) \geq 1 > \Jcal(\pi_1^*, \mu_1^*)$ ruling out the existence of a recovery sequence in this setting. Moreover, for any $\gamma > 0$, the pair $(\pi_3^*, \mu_1^*)$ is an optimal solution to \cref{pr:APPROX-BilevelHitchNoCostReg}, confirming that the solutions of the regularized problem do not converge to a solution of the original problem \cref{pr:PROB-BilevelHitch}.
			
		\item The above counterexample relies on the fact that, \textcolor{cyan}{in certain situations}, the Hitchcock problem \cref{pr:PROB-Hitch} \textcolor{cyan}{admits} multiple optimal solutions. However, if the Hitchcock problem \textcolor{cyan}{only} admits a unique solution, a recovery sequence for some optimal solution $(\pi^*, \mu_1^*)$ to \cref{pr:PROB-BilevelHitch} is given by
		\begin{equation*}
			\mu_{1,k}^* = \mu_1^*,
			\quad  
			c_k^* = c_\drm,
			\quad \text{and} \quad
			\pi_k^* = \Scal_{\gamma_k}(\mu_{1,k}^*, \mu_2^\drm, c_k^*)
		\end{equation*}
		for all $k \in \N$. \textcolor{cyan}{In the above,} $\Scal_{\gamma_k} \colon \R^{n_1} \times \R^{n_2} \times \R^{n_1 \times n_2}$ \textcolor{cyan}{shall} denote the solution operator of the regularized Hitchcock problem \cref{pr:REG-RegHitch}.
		
		\item The arguments from \textcolor{cyan}{both} \cref{lm:APPROX-ClustPointFeas} and \cref{th:APPROX-ClustPointOptimal} also hold for every other cluster point of the sequence of regularized solutions (there may be several).
	\end{enumerate}
\end{remark}

\textcolor{cyan}{Fortunately, the following corollary shows that for any solution to the original bilevel Hitchcock problem, one can always construct a recovery sequence, thereby guaranteeing that solutions to the former can be approximated by solutions to the regularized bilevel Hitchcock problem.}

%{\color{cyan}
\begin{corollary}
\label{co:Approx-ExistRecovSeq}
	For any given optimal solution $(\pi^*, \mu_1^*)$ to the non-regularized bilevel Hitchcock problem \cref{pr:PROB-BilevelHitch}, there exists a recovery sequence in the sense of \cref{th:APPROX-ClustPointOptimal}.
	
	Hence, the cluster point $(\bar{\pi}, \bar{\mu}_1)$ found in \cref{lm:APPROX-ClustPointFeas} is an optimal solution to the non-regularized bilevel Hitchcock problem \cref{pr:PROB-BilevelHitch}.
\end{corollary}
%}

%{\color{cyan}
\begin{proof}
	Because $\pi^* \geq 0$ solves \cref{pr:PROB-Hitch} \wrt $\mu_1^*$, $\mu_2^\drm$, as well as $c_{\drm}$ and because strong duality holds, we can find a dual solution $(\alpha_1^*, \alpha_2^*) \in \R^{n_1} \times \R^{n_2}$ which satisfies
	\begin{equation*}
		\alpha_1^* \oplus \alpha_2^* \leq c_{\drm}
		\qquad \text{as well as} \qquad
		\alpha_{1, i_1}^* + \alpha_{2, i_2}^*
		= (c_{\drm})_{i_1, i_2}
		~ \, \text{wherever} ~ \, \pi_{i_1, i_2}^* > 0,
	\end{equation*}
	see \eg \cite[Theorem 2]{flood1953hitchcock}. If we define the cost $c_k^* \coloneqq c_{\drm} - \gamma_k \pi^*$, $k \in \N$, we immediately receive that
	\begin{equation*}
		\pi_k^*
		\coloneqq \frac1{\gamma_k} (\alpha_1^* \oplus \alpha_2^* - c_k^*)_+
		= \frac1{\gamma_k} (\alpha_1^* \oplus \alpha_2^* - c_{\drm} + \gamma_k \pi^*)_+
		= \pi^*.
	\end{equation*}
	For any $k \in \N$, \cref{th:REG-RegHitchDualSys} therefore yields that $\pi_k^* = \pi^*$ solves the regularized Hitchcock problem \textnormal{\hyperref[pr:REG-RegHitch]{(H$_{\gamma_k}$)}} \wrt $\mu_1^*$, $\mu_2^\drm$, and $c_k^*$. Consequently, the sequence $(\pi_k^*, \mu_{1,k}^*, c_k^*)_{k \in \N}$, where $\mu_{1, k}^* \coloneqq \mu_1^*$, is feasible for the sequence of regularized bilevel problems \textnormal{\hyperref[pr:REG-RegBilevelHitch]{(BH$_{\gamma_k}$)$_{k \in \N}$}}. The second condition from the definition of a recovery sequence in \cref{th:APPROX-ClustPointOptimal} readily follows by construction of $c_k^*$ and from the fact that the above sequence is actually constant in $\pi_k^*$ and $\mu_{1,k}^*$.
\end{proof}

\section{A Further Regularization of the (Dual) Hitchcock Problem}\label{sc:DERIV}
In \cref{sc:PROB}, we have seen that we can approximate solutions to the bilevel Hitchcock problem \cref{pr:PROB-BilevelHitch} by solutions to the regularized bilevel Hitchcock problems \cref{pr:REG-RegBilevelHitch}, if we drive the regularization parameter $\gamma$ towards $0$. However, until now we did not comment on how to actually compute solutions to the latter.

In general, the problems \hyperref[pr:REG-RegBilevelHitch]{(BH$_\gamma$)$_{\gamma > 0}$} are non-convex optimization problems, which inherently present several challenges, such as the presence of non-global minima and/or saddle points. Additionally, although there exists a solution operator \textcolor{cyan}{which maps} the marginals to the (unique) solution of the regularized Kantorovich problem, we currently lack a first-order optimality system \textcolor{cyan}{and} derivatives of the solution operator that could be used to solve the regularized bilevel problems. While obtaining such first-order information might be feasible in certain related problems (see, e.g., \cite{carioni2023extremal}), we pursue a different approach that we consider more convenient.

The solution to the regularized Hitchcock problems \hyperref[pr:REG-RegHitch]{(H$_\gamma$)$_{\gamma > 0}$} is unique, but its corresponding Lagrangian multipliers are not: \cref{th:REG-RegHitchDualSys} reveals that we can (constantly) shift the multipliers in opposite directions and these shifted vectors will still be multipliers for the regularized transport plan; also, the kernel of the $(\substitute)_+$-operator provides another source of non-uniqueness for the multipliers, as any value of their outer sum that is smaller than the cost of transportation is simply cut off. While the uniqueness of the dual variables may not be necessary to compute derivatives of \cref{pr:REG-RegHitch}'s solution mapping, we shall see in the following that it turns out to be very useful if we enforce \textcolor{cyan}{this property} by introducing another regularization.

Given the marginals $\mu_1 \in \R^{n_1}$ and $\mu_2 \in \R^{n_2}$, the cost matrix $c \in \R^{n_1 \times n_2}$, as well as the regularization parameters $\gamma, \varepsilon > 0$, we seek to find solutions $\alpha_1 \in \R^{n_1}$ and $\alpha_2 \in \R^{n_2}$ to the regularized nonlinear system of equations
\begin{align}
\label{eq:DERIV-DualNonlinSyst}
	\begin{split}
		(\alphasum)_+^{\phantomtop} \one + \gamma \varepsilon \alpha_1
		&= \gamma \mu_1, \\
		(\alphasum)_+^\top \one + \gamma \varepsilon \alpha_2
		&= \gamma \mu_2.
	\end{split}
\end{align}
The above system (note the similarity between \cref{eq:REG-DualSystem} and \cref{eq:DERIV-DualNonlinSyst}) corresponds to the first-order necessary and sufficient optimality conditions of a standard Tikhonov regularization of the dual problem \cref{pr:REG-DualRegHitch}, namely,
\begin{equation}
\label{pr:REG-DualRegRegHitch}
\tag{D$_\gamma^\varepsilon$}
	\begin{array}{rl}
		\sup\limits_{\alpha_1, \alpha_2} & \scalarproduct{\alpha_1}{\mu_1} + \scalarproduct{\alpha_2}{\mu_2}- \frac{1}{2\gamma} \norm[F]{(\alphasum)_+}^2 - \frac\varepsilon2 \bigl( \norm{\alpha_1}^2 + \norm{\alpha_2}^2 \bigr) \\[0.75em]
		\text{s.t.} & \alpha_i \in \R^{n_i}, ~ i = 1, 2.
	\end{array}
\end{equation}
Because \cref{pr:REG-DualRegRegHitch}'s objective function, which we may denote by $\Psi(\alpha_1, \alpha_2)$, is strongly concave \textcolor{cyan}{as well as} differentiable and has the gradient
\begin{equation*}
	\nabla \Psi(\alpha_1, \alpha_2)
	=
	\begin{pmatrix}
		\mu_1 - \frac1\gamma (\alphasum)_+^{\phantomtop} \one - \varepsilon \alpha_1 \\[0.5em]
		\mu_2 - \frac1\gamma (\alphasum)_+^\top \one - \varepsilon \alpha_2
	\end{pmatrix}.
\end{equation*}
\textcolor{cyan}{$\Psi$}'s unique maximum $(\alpha_1^*, \alpha_2^*)$ solves the above nonlinear system. Note that \cref{eq:DERIV-DualNonlinSyst} always admits a unique solution, independently of the marginal's signs and their masses!

This guarantees the existence of the solution operator
\begin{equation*}
	\Fcal_{\gamma, \varepsilon} \colon \R^{n_1} \times \R^{n_2} \to \R^{n_1} \times \R^{n_2},
	\quad (\mu_1, \mu_2) \mapsto (\alpha_1, \alpha_2),
\end{equation*}
with $\alpha_1$ and $\alpha_2$ solving \cref{eq:DERIV-DualNonlinSyst} \wrt $\mu_1$ and $\mu_2$ (as well as $c$, $\gamma$, and $\varepsilon$). $\Fcal_{\gamma, \varepsilon}$ is a bijective mapping and its inverse is given by
\begin{equation*}
	\Fcal_{\gamma, \varepsilon}^{-1}(\alpha_1, \alpha_2)
	= \frac1\gamma \bigl( (\alphasum)_+ \one + \gamma \varepsilon \alpha_1, (\alphasum)_+^\top \one + \gamma \varepsilon \alpha_2 \bigr),
\end{equation*}
Moreover, $\Fcal_{\gamma, \varepsilon}$ is a Lipschitz continuous mapping, see \eg \cite[Proposition 2G.4]{dontchev2009implicit}.

%\begin{definition}
%	We denote the solution operator of the nonlinear system of equations \cref{eq:DERIV-DualNonlinSyst} by
%	\begin{equation*}
%		\Fcal_{\gamma, \varepsilon} \colon \R^{n_1} \times \R^{n_2} \to \R^{n_1} \times \R^{n_2},
%		\quad (\mu_1, \mu_2) \mapsto (\alpha_1, \alpha_2),
%	\end{equation*}
%	where $\alpha_1$ and $\alpha_2$ solve \cref{eq:DERIV-DualNonlinSyst} \wrt $\mu_1$, $\mu_2$, $\gamma$, and $\varepsilon$.
%\end{definition}

\textcolor{cyan}{In the following,} we are interested in the properties of a mapping which maps the marginals onto the corresponding (optimal) transport plan. \textcolor{cyan}{We will use this mapping to} replace the lower-level Kantorovich problem in the formulation of the bilevel Hitchcock problem in \cref{sc:REDU}.

\begin{definition}
\label{df:DERIV-RegMTPMap}
	The \emph{regularized marginal-to-transport-plan mapping} is given by
	\begin{equation*}
		\Scal_{\gamma, \varepsilon} \colon \R^{n_1} \times \R^{n_2} \to \R^{n_1 \times n_2},
		\quad (\mu_1, \mu_2) \mapsto \frac1\gamma (\alphasum)_+,
	\end{equation*}
	where $(\alpha_1, \alpha_2) = \Fcal_{\gamma, \varepsilon}(\mu_1, \mu_2)$.
\end{definition}

The above definition gives rise to several remarks.

\begin{remark}
	\begin{itemize}
		\item The term ``marginal-to-transport-plan mapping'' alone may be misleading in the sense that $\pi_{\gamma, \varepsilon} \coloneqq \Scal_{\gamma, \varepsilon}(\alpha_1, \alpha_2)$ is in general not a regularized optimal transport plan (\ie not a solution to \cref{pr:REG-RegHitch}) \wrt $\mu_1$, $\mu_2$, and $c$. It is important to keep in mind that the mapping $\Scal_{\gamma, \varepsilon}$ arises from a Tikhonov regularization of the dual problem. Nevertheless, with the same arguments as in \cite[Section 2.3]{lorenz2021quadratically} one can show that, for $\varepsilon \searrow 0$, $\pi_{\gamma, \varepsilon}$ converges to $\pi_\gamma$, the unique solution to \cref{pr:REG-RegHitch} \wrt $\mu_1$, $\mu_2$, and $c$, which justifies the name we have coined for the mapping from \cref{df:DERIV-RegMTPMap}.
		
		\item To ease the computation of the (sub-)gradients and because the focus of this section does not lie on the approximability of non-regularized solutions but on the computation of regularized solutions, we remove the cost matrix from the set of optimization variables. However, we expect the results to be similar, if one decides to keep the set of optimization variables from the previous subsections.
		
		\item Moreover, we expect that the further regularization of the dual problem behaves well in view of the approximation results from \cref{sc:REG}, \ie that the results of \cref{th:REG-RegBilevelHitchSolu}, \cref{lm:APPROX-ClustPointFeas}, and \cref{th:APPROX-ClustPointOptimal} hold in the case of \cref{pr:REDU-TwiceRegBilevelHitch} from \cref{sc:REDU}.
		
		\item Because the mapping $\Fcal_{\gamma, \varepsilon}$ and both the $\oplus$-operator as well as the $(\substitute)_+$-operator are Lipschitz continuous, $\Scal_{\gamma, \varepsilon}$ too is a Lipschitz continuous mapping.
	\end{itemize}
\end{remark}

% % % % % % % % % % % % % % % % % % % % % % % % % % % % % % % % % % % % % % %

\subsection{(Sub-)Gradient Analysis}
\label{sc:SGA}

In the remainder of this section, we will: characterize the points at which $\Scal_{\gamma, \varepsilon}$ is differentiable; calculate its derivative at those points; and, for the points of non-differentiability, find a manageable representation of its Bouligand subdifferential. The following definition addresses the latter aspect.

\begin{definition}
\label{df:DERIV-BoulSubdiff}
	Let $f \colon \R^m \to \R^n$, $m, n \in \N$, be a locally Lipschitz function. By Rademacher's theorem, $f$ is differentiable on a set $\Dcal_f$, whose complement is a Lebesgue null set in $\R^m$. The set
	\begin{equation*}
		\partial_B f(x)
		\coloneqq \Bigl\{
		\lim_{k \to \infty} f'(x_k) \colon
		(x_k)_{k \in \N} \subset \Dcal_f, ~ x_k \to x ~ \text{as} ~ k \to \infty
		\Bigr\}
	\end{equation*}
	is called the \emph{Bouligand subdifferential} of $f$ at some point $x \in \R^m$. It relates to Clarke's generalized Jacobian $\partial f(x)$ via the definition
		\begin{equation*}
			\partial f(x) \coloneqq \operatorname{co} \bigl( \partial_B f(x) \bigr),
		\end{equation*}
		see \eg \cite[Section 2.6]{clarke1990optimization}. Note that this set is closed qua definition, see \eg \cite[Proposition 2.6.1]{clarke1990optimization}.
\end{definition}

%\begin{lemma}
%\label{lm:DERIV-RegMTPMapLipschitz}
%	The regularized marginal-to-transport-plan mapping is globally Lipschitz continuous and therefore differentiable almost everywhere on $\R^{n_1} \times \R^{n_2}$.
%\end{lemma}
%
%\begin{proof}
%	We first note that $\Fcal_{\gamma, \varepsilon}$ is bijective: its injectivity follows from the uniqueness of the solution to \cref{eq:DERIV-DualNonlinSyst}; it is surjective, because the preimage of any $(\alpha_1, \alpha_2) \in \R^{n_1} \times \R^{n_2}$ can be computed by evaluating the left-hand side of the equations in \cref{eq:DERIV-DualNonlinSyst}. The inverse of $\Fcal_{\gamma, \varepsilon}$ is given by
%	\begin{equation*}
%		\Fcal_{\gamma, \varepsilon}^{-1}(\alpha_1, \alpha_2)
%		= \frac1\gamma \bigl( (\alphasum)_+ \one + \gamma \varepsilon \alpha_1, (\alphasum)_+^\top \one + \gamma \varepsilon \alpha_2 \bigr),
%	\end{equation*}
%	which is continuous and piecewise linear with finitely many pieces. This immediately yields the Lipschitz continuity of $\Fcal_{\gamma, \varepsilon}$. The Lipschitz continuity of $\Scal_{\gamma, \varepsilon}$ then follows from the (sub-)linearity of $(\substitute)_+$ and $\oplus$.
%\end{proof}

In order to be able to compute the derivative of the regularized marginal-to-transport-plan mapping $\Scal_{\gamma, \varepsilon}$, we first need to characterize the directional derivative of the solution operator of the nonlinear system \cref{eq:DERIV-DualNonlinSyst}.

\begin{lemma}
	The solution operator $\Fcal_{\gamma, \varepsilon}$ is Hadamard differentiable at any point $\mu = (\mu_1, \mu_2) \in \R^{n_1} \times \R^{n_2}$ and in any direction $h = (h_1, h_2) \in \R^{n_1} \times \R^{n_2}$ and its directional derivative is given by the unique solution $\Fcal_{\gamma, \varepsilon}'(\mu; h) = (\eta_1, \eta_2) \in \R^{n_1} \times \R^{n_2}$ to the nonlinear system of equations
	\begin{align}
	\label{eq:DERIV-SolOpDirecDeriv}
		\begin{split}
			{\max}'(\alphasum; \etasum)^{\! \phantomtop} \! \one + \gamma \varepsilon \eta_1
			&= \gamma h_1, \\
			{\max}'(\alphasum; \etasum)^{\! \top} \! \one + \gamma \varepsilon \eta_2
			&= \gamma h_2.
		\end{split}
	\end{align}
	In the above, $(\alpha_1, \alpha_2) = \Fcal_{\gamma, \varepsilon}(\mu)$ and
	\begin{equation}
	\label{eq:DERIV-MaxOpDirecDeriv}
		{\max}'(a; b)
		= \begin{cases}
			b, & \text{if} ~ a > 0, \\
			\max \{0, b\}, & \text{if} ~ a = 0, \\
			0, & \text{if} ~ a < 0,
		\end{cases}
	\end{equation}
	denotes the directional derivative of the mapping $x \mapsto \max \{0, x\}$, which is understood to be applied entry-wise to the matrices $\alphasum$ and $\etasum$.
\end{lemma}

\begin{proof}
	For any $t > 0$, we consider the point $\alpha_t = \Fcal_{\gamma, \varepsilon}(\mu + th)$ and the difference quotient $\eta_t = {}^{1}\!/\!{}_{t} (\alpha_t - \alpha)$, which satisfies
	\begin{align*}
		\frac{\bigl( (\alpha_{1,t} \oplus \alpha_{2,t} - c)_+ - (\alphasum)_+ \bigr)^{\! \phantomtop} \! \one}{t} + \gamma \varepsilon \eta_{1,t}
		&= \gamma h_1, \\
		\frac{\bigl( (\alpha_{1,t} \oplus \alpha_{2,t} - c)_+ - (\alphasum)_+ \bigr)^{\! \top} \! \one}{t} + \gamma \varepsilon \eta_{2,t}
		&= \gamma h_2.
	\end{align*}
	Because $\Fcal_{\gamma, \varepsilon}$ is Lipschitz continuous, the sequence $(\eta_t)_{t \searrow 0}$ is bounded and converges (up to subsequences) to some point $\eta$. Because the mapping $x \mapsto \max \{0, x\}$ is Hadamard differentiable and $\alpha_t = \alpha + t \eta + o(t)$, we can pass to the limit in the above system of equations to arrive at 
	\begin{align}
	\label{eq:DERIV-DiffQuoLimSyst}
		\begin{split}
			{\max}'(\alphasum; \etasum)^{\! \phantomtop} \! \one + \gamma \varepsilon \eta_1
			&= \gamma h_1, \\
			{\max}'(\alphasum; \etasum)^{\! \top} \! \one + \gamma \varepsilon \eta_2
			&= \gamma h_2.
		\end{split}
	\end{align}
	
	It is straightforward to check that the operator
	\begin{equation*}
		(u_1, u_2) \mapsto \bigl( {\max}'(\alphasum; u_1 \oplus u_2) \one + \gamma \varepsilon u_1, {\max}'(\alphasum; u_1 \oplus u_2)^\top \one + \gamma \varepsilon u_2 \bigr)
	\end{equation*}
	is strongly monotone, coercive, and continuous in the sense of \cite{ruzicka2004nichtlineare}. Consequently, by the Browder-Minty theorem, see \eg \cite[Satz 1.5]{ruzicka2004nichtlineare}, the system from \cref{eq:DERIV-DiffQuoLimSyst} is uniquely solvable and the entire sequence $(\eta_t)_{t \searrow 0}$ converges to $\eta$, which proves the directional differentiability of $\Fcal_{\gamma, \varepsilon}$. That $\Fcal_{\gamma, \varepsilon}$ is moreover Hadamard differentiable follows from its Lipschitz continuity.
\end{proof}

\begin{proposition}
\label{pp:DERIV-RegMTPMapDirecDeriv}
	The regularized marginal-to-transport-plan mapping $\Scal_{\gamma, \varepsilon}$ is Ha\-da\-mard differentiable and its directional derivative at the point $\mu = (\mu_1, \mu_2) \in \R^{n_1} \times \R^{n_2}$ in the direction $h = (h_1, h_2) \in \R^{n_1} \times \R^{n_2}$ is given by
	\begin{equation*}
		\Scal_{\gamma, \varepsilon}'(\mu; h)
		= \frac1\gamma {\max}'(\alphasum; \etasum),
		\quad \text{where} ~ (\eta_1, \eta_2) = \Fcal_{\gamma, \varepsilon}'(\mu; h).
	\end{equation*}
\end{proposition}

\begin{proof}
	One easily checks that the mapping $\Pcal_{\gamma} \colon \R^{n_1} \times \R^{n_2} \to \R^{n_1 \times n_2}$, $(u_1, u_2) \mapsto \frac1\gamma (u_1 \oplus u_2 - c)_+$ is Hadamard differentiable with directional derivative
	\begin{equation*}
		\Pcal_{\gamma}'(u; h)
		= \frac1\gamma {\max}'(u_1 \oplus u_2 - c; h_1 \oplus h_2)
		\quad \text{for all} ~ u, h \in \R^{n_1} \times \R^{n_2}.
	\end{equation*}
	The claim then follows from an application of the Hadamard chain rule, see \eg \cite[Proposition 3.6]{shapiro1990concepts}, to the mapping $\Scal_{\gamma, \varepsilon} = \Pcal_{\gamma} \circ \Fcal_{\gamma, \varepsilon}$.
\end{proof}

We define the following sets to characterize the points at which $\Scal_{\gamma, \varepsilon}$ is not only directional differentiable but (totally) differentiable:

\begin{definition}
\label{df:DERIV-Omega+Omega0Omega-}
	Given some point $\mu \in \R^{n_1} \times \R^{n_2}$ and $\alpha = \Fcal_{\gamma, \varepsilon}(\mu)$, we define
	\begin{align*}
		\Omega_+(\mu)
		&\coloneqq \{(i_1, i_2) \in \Omega \colon (\alphasum)_{i_1, i_2} > 0\}, \\
		\Omega_0^{\hphantom{-}}(\mu)
		&\coloneqq \{(i_1, i_2) \in \Omega \colon (\alphasum)_{i_1, i_2} = 0\}, \\
		\Omega_-(\mu)
		&\coloneqq \{(i_1, i_2) \in \Omega \colon (\alphasum)_{i_1, i_2} < 0\}.
	\end{align*}
	If there is no risk of confusion, we refrain from explicitly mentioning the dependence of the sets on the point $\mu$. Note that $\Omega = \Omega_+ ~ \dot{\cup} ~ \Omega_0 ~ \dot{\cup} ~ \Omega_-$.
\end{definition}

With the definitions from above, we can now precisely characterize the points at which $\Scal_{\gamma, \varepsilon}$ is (totally) differentiable.

\begin{proposition}
\label{pp:DERIV-RegMTPMapDiffPoints}
	The regularized marginal-to-transport-plan mapping $\Scal_{\gamma, \varepsilon}$ is differentiable at $\mu \in \R^{n_1} \times \R^{n_2}$ if and only if $\Omega_0(\mu) = \emptyset$.
\end{proposition}

\begin{proof}
	Because $\Scal_{\gamma, \varepsilon}$ is Lipschitz continuous, it is sufficient to show that $\Omega_0 = \emptyset$ if and only if $\Scal_{\gamma, \varepsilon}'(\mu; \substitute)$ is linear.
	
	On the one hand, if $\Omega_0 = \emptyset$, then $\Fcal_{\gamma, \varepsilon}'(\mu; h)$ is linear \wrt $h$ and $\frac1\gamma {\max}'(\alphasum; s_1 \oplus s_2)$ is linear \wrt $s = (s_1, s_2)$, implying the linearity of $\Scal_{\gamma, \varepsilon}'(\mu; \substitute)$.
	
	On the other hand, if $\Scal_{\gamma, \varepsilon}'(\mu; \substitute)$ is linear, then we find that
	\begin{align*}
		0 &= \Scal_{\gamma, \varepsilon}'(\mu; h) + \Scal_{\gamma, \varepsilon}'(\mu; -h) \\
		&= \frac1\gamma
		\begin{cases}
			\eta_1^{i_1} + \eta_2^{i_2} + \theta_1^{i_1} + \theta_2^{i_2}, & \text{if} ~ (i_1, i_2) \in \Omega_+, \\
			\max\{0, \eta_1^{i_1} + \eta_2^{i_2}\} + \max\{0, \theta_1^{i_1} + \theta_2^{i_2}\}, & \text{if} ~ (i_1, i_2) \in \Omega_0, \\
			0, & \text{if} ~ (i_1, i_2) \in \Omega_+,
		\end{cases}
	\end{align*}
	for arbitrary $h$ with $\eta = (\eta_1, \eta_2) = \Fcal_{\gamma, \varepsilon}'(\mu; h)$ and $\theta = (\theta_1, \theta_2) = \Fcal_{\gamma, \varepsilon}'(\mu; -h)$. In particular,
	\begin{equation}
	\label{eq:DERIV-SolOpDirecDerivNonpos}
		\eta_1^{i_1} + \eta_2^{i_2} \leq 0
		\quad \text{and} \quad
		\theta_1^{i_1} + \theta_2^{i_2} \leq 0
		\quad \text{for all} ~ (i_1, i_2) \in \Omega_0.
	\end{equation}
	However, using the bijectivity of the solution map $\Fcal_{\gamma, \varepsilon}$ it is easy to construct a direction $\tilde{h}$ such that $\tilde{\eta} = \Fcal_{\gamma, \varepsilon}'(\mu; \tilde{h})$ contradicts \cref{eq:DERIV-SolOpDirecDerivNonpos}. Hence $\Omega_0$ must be empty.
\end{proof}

In the following, we denote the set of points at which $\Scal_{\gamma, \varepsilon}$ is differentiable by $\Dcal_{\Scal_{\gamma, \varepsilon}}$, \ie
\begin{equation*}
	\Dcal_{\Scal_{\gamma, \varepsilon}} 
	= \{
		\mu \in \R^{n_1} \times \R^{n_2}
		\colon \Omega_0(\mu) = \emptyset
	\}.
\end{equation*}
To be able to write the derivatives and (Bouligand) subgradients of $\Scal_{\gamma, \varepsilon}$ in a compact form, we first need to establish some notation.

\begin{definition}
\label{df:DERIV-CharMatMaskOpNewtonMat}
	Let $\Acal \subset \Omega$ be an arbitrary index set. Then, we define
	\begin{enumerate}
		\item the \emph{characteristic matrix} $\chi(\Acal) \in \R^{n_1 \times n_2}$ of the set $\Acal$ by
		\begin{equation*}
			\chi(\Acal)_{i_1, i_2}
			\coloneqq \begin{cases}
				1, & \text{if} ~ (i_1, i_2) \in \Acal, \\
				0, & \text{else;}
			\end{cases}
		\end{equation*}
		
		\item the \emph{masking operator} $\Mcal(\Acal) \colon \R^{n_1 \times n_2} \to \R^{n_1 \times n_2}$ associated with $\Acal$ as an entrywise multiplication with the characteristic matrix, \ie
		\begin{equation*}
			\Mcal(\Acal)(M)
			\coloneqq \bigl( \chi(\Acal)_{i_1, i_2} M_{i_1, i_2} \bigr)_{(i_1, i_2) \in \Omega};
		\end{equation*}
		
		\item\label{df:DERIV-CharMatMaskOpNewtonMat-SysMat} the \emph{system matrix} $\Ncal(\Acal) \in \R^{(n_1+n_2) \times (n_1+n_2)}$ associated with $\Acal$ by
			\begin{equation*}
				\Ncal(\Acal)
				\coloneqq \begin{pmatrix}
					\diag \bigl( \chi(\Acal) \one \bigr) & \chi(\Acal) \\
					\chi(\Acal)^\top & \diag \bigl( \chi(\Acal)^\top \one \bigr)
				\end{pmatrix}.
			\end{equation*}
	\end{enumerate}
	Let $\Bcal \subset \Omega$ be another index set such that $\Acal \subset \Bcal$. Then, we say that
	\begin{enumerate}
		\item[4.] $\Acal$ has an \emph{outer structure} \wrt $\Bcal$, if there exist vectors $v_1 \in \R^{n_1}$ and $v_2 \in \R^{n_2}$ such that
		\begin{equation*}
			(v_1 \oplus v_2)_{\Acal} > 0
			\quad \text{and} \quad
			(v_1 \oplus v_2)_{\Bcal \setminus \Acal} < 0.
		\end{equation*}
		The above notation means that the entries of the matrix $v_1 \oplus v_2$ shall be strictly positive and strictly negative for all indices belonging to the index sets $\Acal$ and $\Bcal \setminus \Acal$, respectively.
	\end{enumerate}
\end{definition}

\begin{remark}
	A certain instance of the matrix $\Ncal(\Acal)$ from \cref{df:DERIV-CharMatMaskOpNewtonMat}.\cref{df:DERIV-CharMatMaskOpNewtonMat-SysMat} also plays an important role in \cite[Section 3]{lorenz2021quadratically}. Therein, the authors identify the matrix $\Ncal \bigl( \{(i_1, i_2) \colon \alpha_1^{i_1} + \alpha_2^{i_2} - c_{i_1, i_2} \geq 0\} \bigr)$ to be a Newton derivative of the semismooth mapping
	\begin{equation*}
		F(\alpha_1, \alpha_2)
		\coloneqq \begin{pmatrix}
			(\alphasum)_+^{\phantomtop} \one - \gamma \mu_1 \\[0.5em]
			(\alphasum)_+^\top \one - \gamma \mu_2
		\end{pmatrix}.
	\end{equation*}
	This mapping describes the regularized Hitchcock problem's first order optimality conditions from \cref{th:REG-RegHitchDualSys}. The authors use this Newton derivative to implement a semismooth Newton method to solve the regularized Hitchcock problem.
\end{remark}

Let us begin with the characterization of the derivative of $\Scal_{\gamma, \varepsilon}$ at the points at which it is differentiable.

\begin{theorem}
\label{th:DERIV-RegMTPMapDeriv}
	If $\mu \in \Dcal_{\Scal_{\gamma, \varepsilon}}$ is a point where $\Scal_{\gamma, \varepsilon}$ is differentiable, then
	\begin{equation*}
		\Scal_{\gamma, \varepsilon}'(\mu)
		= \Mcal(\Omega_+(\mu)) \circ \oplus \circ \bigl( \Ncal(\Omega_+(\mu)) + \gamma \varepsilon I \bigr)^{-1}
	\end{equation*}
	Here, $I$ refers to the $(n_1+n_2)$-dimensional identity matrix.
\end{theorem}

\begin{proof}
	Let $h = (h_1, h_2)$ be an arbitrary direction. If $\mu$ is a point where $\Scal_{\gamma, \varepsilon}$ is differentiable, then $\Omega_0 = \emptyset$ and the directional derivative $\eta = (\eta_1, \eta_2) = \Fcal_{\gamma, \varepsilon}'(\mu; h)$ of the solution operator $\Fcal_{\gamma, \varepsilon}$ satisfies
	\begin{align*}
		&\sum_{i_2 \colon (i_1, i_2) \in \Omega_+} \bigl( \eta_1^{i_1} + \eta_2^{i_2} \bigr) + \gamma \varepsilon \eta_1^{i_1}
		= \gamma h_1^{i_1}
		\quad \text{for all} ~ i_1 \in \Omega_1, \\
		&\sum_{i_1 \colon (i_1, i_2) \in \Omega_+} \bigl( \eta_1^{i_1} + \eta_2^{i_2} \bigr) + \gamma \varepsilon \eta_2^{i_2}
		= \gamma h_2^{i_2}
		\quad \text{for all} ~ i_2 \in \Omega_2.
	\end{align*}
	With the definitions from \cref{df:DERIV-CharMatMaskOpNewtonMat}, this can equivalently be written as
	\begin{equation*}
		\bigl( \Ncal(\Omega_+) + \gamma \varepsilon I \bigr) \eta
		= \gamma h.
	\end{equation*}
	By construction, the matrix $\Ncal(\Omega_+)$ is nonnegative, symmetric, and diagonally dominant, hence positive semidefinite. The directional derivative $\eta$ therefore takes the form
	\begin{equation*}
		\eta
		= \gamma \bigl( \Ncal(\Omega_+) + \gamma \varepsilon I \bigr)^{-1} h,
	\end{equation*}
	implying that
	\begin{equation*}
		\Fcal_{\gamma, \varepsilon}'(\mu) = \gamma \bigl( \Ncal(\Omega_+) + \gamma \varepsilon I \bigr)^{-1}.
	\end{equation*}
	Consequently, the directional derivative from \cref{pp:DERIV-RegMTPMapDirecDeriv} can be written as
	\begin{equation*}
		\Scal_{\gamma, \varepsilon}'(\mu; h)
		= \frac1\gamma \bigl( \Mcal(\Omega_+) \circ \oplus \circ \Fcal_{\gamma, \varepsilon}'(\mu) \bigr) h,
	\end{equation*}
	which yields the claim.
\end{proof}

Now, we consider the points at which $\Scal_{\gamma, \varepsilon}$ is not differentiable. \textcolor{cyan}{For these points, we can find a precise characterization of $\Scal_{\gamma, \varepsilon}$'s  Bouligand subdifferential.}

\begin{theorem}
\label{th:DERIV-RegMTPMapBouliSubdiff}
	If $\mu \in (\R^{n_1} \times \R^{n_2}) \setminus \Dcal_{\Scal_{\gamma, \varepsilon}}$ is a point where $\Scal_{\gamma, \varepsilon}$ is not differentiable, then
	\begin{align*}
%	\label{eq:DERIV-CharacSubdiff}
		\partial_B \Scal_{\gamma, \varepsilon}(\mu)
		= \Bigl\{
			\Mcal(\Omega_+(\mu) \cup \Acal) \circ \oplus \circ \bigl( \Ncal(\Omega_+(\mu) \cup \Acal) + \gamma \varepsilon I \bigr)^{-1} \colon
			\begin{matrix}
				\Acal ~ \text{has an outer} \\ \text{structure w.r.t.} ~ \Omega_0(\mu)
			\end{matrix}
		\Bigr\}.
	\end{align*}
\end{theorem}

\begin{proof}
	To check the first inclusion, let $G \in \partial_B \Scal_{\gamma, \varepsilon}(\mu)$ be a given Bouligand subgradient of $\Scal_{\gamma, \varepsilon}$ at $\mu$. By \cref{df:DERIV-BoulSubdiff}, there exists a sequence $(\mu_k)_{k \in \N} \subset \Dcal_{\Scal_{\gamma, \varepsilon}}$ such that $\mu_k \to \mu$ as $k \to \infty$ and
	\begin{equation*}
		G = \lim_{k \to \infty} \Scal_{\gamma, \varepsilon}'(\mu_k)
		= \lim_{k \to \infty} \Mcal( \Omega_+(\mu_k)) \circ \oplus \circ \bigl( \Ncal(\Omega_+(\mu_k)) + \gamma \varepsilon I \bigr)^{-1},
	\end{equation*}
	where the second equality stems from \cref{th:DERIV-RegMTPMapDeriv}. By construction, the integer matrices $\Ncal(\Omega_+(\mu_k))$ are bounded. As a consequence, there exists some $K \in \N$ such that $\Omega_+(\mu_k) = \Omega_+^K \coloneqq \Omega_+(\mu_K)$ for all $k \geq K$ and $G = \Mcal(\Omega_+^K) \circ \oplus \circ \bigl( \Ncal(\Omega_+^K) + \gamma \varepsilon I \bigr)^{-1}$. Moreover, because $\Fcal_{\gamma, \varepsilon}$ is (Lipschitz) continuous, there exists another $K \in \N$ such that
	\begin{equation*}
		\Omega_+(\mu) \subset \Omega_+^K,
		\quad \Omega_-(\mu) \subset \Omega_-^K,
		\quad \text{and still} \quad
		G = \Mcal(\Omega_+^K) \circ \oplus \circ \bigl( \Ncal(\Omega_+^K) + \gamma \varepsilon I \bigr)^{-1}.
	\end{equation*}
	Let us set $\Acal \coloneqq \Omega_+^K \setminus \Omega_+(\mu) \subset \Omega$. By \cref{pp:DERIV-RegMTPMapDiffPoints}, 
	\begin{equation*}
		\Omega_+(\mu) ~ \dot{\cup} ~ \Omega_0(\mu) ~ \dot{\cup} ~ \Omega_-(\mu)
		= \Omega
		= \Omega_+^K ~ \dot{\cup} ~ \Omega_-^K
	\end{equation*}
	and therefore $\Acal \subset \Omega_0(\mu)$. Moreover, $\Omega_0(\mu) \setminus \Acal \subset \Omega_-^K$ and
	\begin{align*}
		&(\alpha_{1,K}^{i_1} - \alpha_1^{i_1}) + (\alpha_{2,K}^{i_2} - \alpha_2^{i_2}) \\
		&\quad = (\alpha_{1,K}^{i_1} + \alpha_{2,K}^{i_2} - c_{i_1, i_2}) - (\alpha_1^{i_1} + \alpha_2^{i_2} - c_{i_1, i_2}) ~
		\begin{cases}
			> 0, & \text{if} ~ (i_1, i_2) \in \Acal, \\
			< 0, & \text{if} ~ (i_1, i_2) \in \Omega_0(\mu) \setminus \Acal.
		\end{cases}
	\end{align*}
	Consequently, $\Acal$ has an outer structure \wrt $\Omega_0(\mu)$ so that $G = \Mcal(\Omega_+(\mu) \cup \Acal) \circ \oplus \circ \bigl( \Ncal(\Omega_+(\mu) \cup \Acal) + \gamma \varepsilon I \bigr)^{-1}$ is an element of the set on the right-hand side of the equation in the formulation of the theorem.
	
	To show the converse inclusion, let $\Acal$ have an outer structure \wrt $\Omega_0(\mu) \neq \emptyset$. By definition, there exist $v_1 \in \R^{n_1}$ and $v_2 \in \R^{n_2}$ with $(v_1 \oplus v_2)_{\Acal} > 0$ and $(v_1 \oplus v_2)_{\Omega_0(\mu) \setminus \Acal} < 0$. We set
	\begin{equation*}
		\delta \coloneqq \frac12 \norm[\infty]{v_1 \oplus v_2}^{-1} \min_{(j_1, j_2) \in \Omega_+(\mu) \cup \Omega_-(\mu)} |\alpha_1^{j_1} + \alpha_2^{j_2} - c_{j_1, j_2}| \in \R_{>0}
	\end{equation*}
	and consider the sequence of points defined by $(\alpha_{1,k}, \alpha_{2,k}) \coloneqq (\alpha_1, \alpha_2) + \frac{\delta}{k} (v_1, v_2)$ for all $k \in \N$. Because $\Fcal_{\gamma, \varepsilon}^{-1}$ is continuous,
	\begin{equation*}
		\mu_k \coloneqq \Fcal_{\gamma, \varepsilon}^{-1}(\alpha_{1,k}, \alpha_{2,k}) ~ \xrightarrow[k \to \infty]{} ~ \Fcal_{\gamma, \varepsilon}^{-1}(\alpha_1, \alpha_2) = \mu.
	\end{equation*}
	By construction,
	\begin{align*}
		&\alpha_{1,k}^{i_1} + \alpha_{2,k}^{i_2} - c_{i_1, i_2} \\
		&\quad = (\alpha_1^{i_1} + \alpha_2^{i_2} - c_{i_1, i_2}) + \frac{\delta}{k}(v_1^{i_1} + v_2^{i_2}) ~
		\begin{cases}
			> 0, & \text{if} ~ (i_1, i_2) \in \Omega_+(\mu) \cup \Acal, \\
			< 0, & \text{if} ~ (i_1, i_2) \in \Omega_-(\mu) \cup (\Omega_0(\mu) \setminus \Acal), \\
		\end{cases}
	\end{align*}
	for all $k \in \N$. Thus, $\Omega_+(\mu_k) = \Omega_+(\mu) \cup \Acal$ and $\Omega_-(\mu_k) = \Omega_-(\mu) \cup (\Omega_0(\mu) \setminus \Acal)$ and hence $\mu_k \in \Dcal_{\Scal_{\gamma, \varepsilon}}$ for all $k \in \N$. Therefore, by \cref{th:DERIV-RegMTPMapDeriv},
	\begin{align*}
		&\Mcal(\Omega_+(\mu) \cup \Acal) \circ \oplus \circ \bigl( \Ncal(\Omega_+(\mu) \cup \Acal) + \gamma \varepsilon I \bigr)^{-1} \\
		&\quad = \lim_{k \to \infty} \Mcal(\Omega_+(\mu_k)) \circ \oplus \circ \bigl( \Ncal(\Omega_+(\mu_k)) + \gamma \varepsilon I \bigr)^{-1}
		= \lim_{k \to \infty} \Scal_{\gamma, \varepsilon}'(\mu_k)
	\end{align*}
	is an element of the Bouligand subdifferential of $\Scal_{\gamma, \varepsilon}$ at $\mu$, as claimed.
\end{proof}

\begin{remark}
	\cref{th:DERIV-RegMTPMapBouliSubdiff} implicitly provides a description of the Bouligand subdifferential of $\Scal_{\gamma, \varepsilon}$ for all points $\mu \in \R^{n_1} \times \R^{n_2}$, \ie even for the points where $\Scal_{\gamma, \varepsilon}$ is differentiable. Let $\mu \in \Dcal_{\Scal_{\gamma, \varepsilon}}$ be such a point. Then by \cref{pp:DERIV-RegMTPMapDirecDeriv}, it holds that $\Omega_0(\mu) = \emptyset$ and so that set on the right-hand side of the characterization in \cref{th:DERIV-RegMTPMapBouliSubdiff} reduces to
	\begin{equation*}
		\bigl\{ \Mcal(\Omega_+(\mu)) \circ \oplus \circ \bigl( \Ncal(\Omega_+(\mu)) + \gamma \varepsilon I \bigr)^{-1} \bigr\}
		= \{\Scal_{\gamma, \varepsilon}'(\mu)\},
	\end{equation*}
	see \cref{th:DERIV-RegMTPMapDeriv}.

	Moreover, because $\Fcal_{\gamma, \varepsilon}$ is continuous and $\Omega_0(\mu) = \emptyset$, the set $\Omega_+(\mu)$ is constant in a neighborhood of the point $\mu$.  By the characterization in \cref{th:DERIV-RegMTPMapDeriv}, the same holds true for the derivative $\Scal_{\gamma, \varepsilon}'(\mu)$ which implies that $\Scal_{\gamma, \varepsilon}$ is continuously differentiable in a neighborhood of $\mu$. Consequently,
	\begin{equation*}
		\partial \Scal_{\gamma, \varepsilon}(\mu)
		= \partial_B \Scal_{\gamma, \varepsilon}(\mu)
		= \{\Scal_{\gamma, \varepsilon}'(\mu)\},
	\end{equation*}
	see \eg \cite[Proposition 2.2]{ulbrich2011semismooth}.
\end{remark}

% % % % % % % % % % % % % % % % % % % % % % % % % % % % % % % % % % % % % % %

\subsection{The Reduced Bilevel Hitchcock Problem}\label{sc:REDU}
Let us recall the bilevel Hitchcock problem from the beginning of \cref{sc:PROB}. For \textcolor{cyan}{lower semicontinuous} $\Jcal \colon \R^{n_1 \times n_2} \times \R^{n_1} \to \R \cup \{+\infty\}$ as well as $\mu_2^\drm \in \R^{n_2}$ and $c_\drm \in \R^{n_1 \times n_2}$, the problem is given by
\begin{equation}
\tag{BH}
\label{pr:PROB-BilevelHitch-Recap}
	\begin{array}{rl}
		\displaystyle \inf_{\pi, \mu_1} & \Jcal(\pi, \mu_1) \\[0.5em]
		& \pi \in \R^{n_1 \times n_2}, \quad \mu_1 \in \R^{n_1}, \\[0.5em]
		& \mu_1 \geq 0, \quad \mu_1^\top \one = {\mu_2^\drm}^\top \one, \\[0.5em]
		& \pi ~ \text{solves \cref{pr:PROB-Hitch} w.r.t.} ~ \mu_1, ~ \mu_2^\drm, ~ \text{and} ~ c_\drm.
	\end{array}
\end{equation}
In this problem, we replace the constraint on $\pi$ by the regularized marginal-to-transport mapping to arrive at the problem
\begin{equation*}
\label{pr:REDU-TwiceRegBilevelHitch}
\tag{BH$_\gamma^\varepsilon$}
	\begin{array}{rl}
		\inf\limits_{\pi, \mu_1} & \Jcal(\pi, \mu_1) \\[0.5em]
		\text{s.t.} & \mu_1 \in \R^{n_1}, \quad \mu_1 \geq 0, \quad \mu_1^\top \one = {\mu_2^\drm}^\top \one, \\[0.5em]
		& \pi \in \R^{n_1 \times n_2}, \quad \pi = \Scal_{\gamma, \varepsilon}(\mu_1, \mu_2^\drm),
	\end{array}
\end{equation*}
which is in turn equivalent to the \emph{reduced bilevel Hitchcock problem}
\begin{equation}
\label{pr:REDU-ReduBilevelHitch}
\tag{RBH$_\gamma^\varepsilon$}
	\begin{array}{rl}
		\inf\limits_{\mu_1} & \Jcal \bigl( \Scal_{\gamma, \varepsilon}(\mu_1, \mu_2^\drm), \mu_1 \bigr) \\
		\text{s.t.} & \mu_1 \in \R^{n_1}, \quad \mu_1 \geq 0, \quad \mu_1^\top \one = {\mu_2^\drm}^\top \one.
	\end{array}
\end{equation}
We abbreviate \cref{pr:REDU-ReduBilevelHitch}'s objective by $f_{\gamma, \varepsilon}(\mu_1) \coloneqq \Jcal \bigl( \Scal_{\gamma, \varepsilon}(\mu_1, \mu_2^\drm), \mu_1 \bigr)$ and call this the \emph{reduced target function}.

%{
%	\color{cyan}
%	\begin{remark}
%	\label{rm:REDU-ReguParamsRational}
%		For the rest of the present paper, we assume that the regularization parameters $\gamma > 0$ and $\varepsilon > 0$ are chosen such that $\gamma \varepsilon \in \Q$. This assumption is introduced solely to facilitate the proof of the subsequent result and does not impose any practical limitations on the numerical treatment of the reduced bilevel Hitchcock problem \cref{pr:REDU-ReduBilevelHitch}.
%	\end{remark}
%}

In the case that $\Jcal$ is sufficiently smooth, the composition of $\Jcal$ and $\Scal_{\gamma, \varepsilon}$ is locally Lipschitz (thus differentiable almost everywhere) and bears Clarke subgradients at any point.
\begin{proposition}
	\label{pp:REDU-CompoDiff}
	Let $\Jcal \in C^1(\R^{n_1 \times n_2} \times \R^{n_1})$ be continuously differentiable. Then, the composition $F_{\gamma, \varepsilon}(\mu_1, \mu_2) \coloneqq \Jcal \bigl( \Scal_{\gamma, \varepsilon}(\mu_1, \mu_2), \mu_1 \bigr)$ is locally Lipschitz continuous and differentiable almost everywhere on $\R^{n_1} \times \R^{n_2}$. Moreover, for any point $\mu = (\mu_1, \mu_2) \in \R^{n_1} \times \R^{n_2}$ and every $\Acal \subset \Omega$ that has an outer structure \wrt the set $\Omega_0(\mu)$, an element of the Clarke subdifferential of $F_{\gamma, \varepsilon}$ at $\mu$ is given by
	\begin{equation*}
		g \coloneqq p + \nabla_{\mu_1} \Jcal \bigl( \Scal_{\gamma, \varepsilon}(\mu_1, \mu_2^\drm), \mu_1 \bigr) ~ \in ~ \partial F_{\gamma, \varepsilon}(\mu),
	\end{equation*}
	where
	\begin{equation*}
		p
		\coloneqq \bigl( \Ncal(\Omega_+(\mu) \cup \Acal ) + \gamma \varepsilon I \bigr)^{-1}
		\begin{pmatrix}
			M^{\phantomtop} \one \\
			M^\top \one
		\end{pmatrix} ~ \in \R^{n_1} \times \R^{n_2}
	\end{equation*}
	and
	\begin{equation*}
		M \coloneqq \Mcal \bigl( \Omega_+(\mu) \cup \Acal \bigr) \nabla_\pi \Jcal \bigl( \Scal_{\gamma, \varepsilon}(\mu), \mu_1 \bigr)
		~ \in \R^{n_1 \times n_2}.
	\end{equation*}
\end{proposition}

\begin{proof}
	The local Lipschitz continuity and thus the almost everywhere differentiability of $F_{\gamma, \varepsilon}$ are obvious from the properties of $\Jcal$ and $\Scal_{\gamma, \varepsilon}$.
	
	To prove the remaining statement, we first define the mapping $\Gcal_{\gamma, \varepsilon}(\mu_1, \mu_2) \coloneqq \bigl( \Scal_{\gamma, \varepsilon}(\mu_1, \mu_2), \mu_1 \bigr)$. By the chain rule for Clarke's generalized gradients (see \eg \cite[Theorem 2.6.6]{clarke1990optimization}), the Clarke subdifferential of the mapping $F_{\gamma, \varepsilon}$ at $\mu$ is given by
	\begin{equation*}
		\partial \bigl( \Jcal \circ \Gcal_{\gamma, \varepsilon} \bigr)(\mu)
		= \nabla_\pi \Jcal(\Gcal_{\gamma, \varepsilon}(\mu)) \, \partial \Scal_{\gamma, \varepsilon}(\mu) + \bigl( \nabla_{\mu_1} \Jcal(\Gcal_{\gamma, \varepsilon}(\mu))^\top, \zero^\top \bigr).
	\end{equation*}
	For any $G \in \partial_B \Scal_{\gamma, \varepsilon}(\mu) \subset \partial \Scal_{\gamma, \varepsilon}(\mu)$ and any $u = (u_1, u_2) \in \R^{n_1} \times \R^{n_2}$, we find that
	\begin{equation*}
		\nabla_\pi \Jcal(\Gcal_{\gamma, \varepsilon}(\mu)) G u
		= \scalarproduct[\R^{n_1} \times \R^{n_2}]{G^* \nabla_\pi \Jcal(\Gcal_{\gamma, \varepsilon}(\mu))}{u},
	\end{equation*}
	where $G^*$ denotes the adjoint of the linear operator $G \colon \R^{n_1} \times \R^{n_2} \to \R^{n_1 \times n_2}$. Let $\Acal \subset \Omega_0(\mu)$ be the \textcolor{cyan}{index} set that realizes $G$, \ie
	\begin{equation*}
		G = \Mcal(\Omega_+(\mu) \cup \Acal) \circ \oplus \circ \bigl(\Ncal(\Omega_+(\mu) \cup \Acal) + \gamma \varepsilon I \bigr)^{-1},
	\end{equation*}
	see \cref{th:DERIV-RegMTPMapBouliSubdiff}. Both $\Mcal(\Omega_+(\mu) \cup \Acal)$ and $\bigl(\Ncal(\Omega_+(\mu) \cup \Acal) + \gamma \varepsilon I \bigr)^{-1}$ are self-adjoint (the latter is symmetric) and the adjoint of the $\oplus$-operator is given by 
	\begin{equation*}
		\oplus^* \colon \R^{n_1 \times n_2} \to \R^{n_1} \times \R^{n_2},
		\quad M \mapsto (\Sigma_1 M, \Sigma_2 M) = (M \one, M^\top \one),
	\end{equation*}
	see the proof of \cref{th:REG-RegHitchDualSys}. Therefore,
	\begin{align*}
		g &\coloneqq G^* \nabla_\pi \Jcal(\Gcal_{\gamma, \varepsilon}(\mu)) + \bigl( \nabla_{\mu_1} \Jcal(\Gcal_{\gamma, \varepsilon}(\mu))^\top, \zero^\top \bigr) \\
		&= \bigl( (\Ncal(\Omega_+(\mu) \cup \Acal) + \gamma \varepsilon I )^{-1} \circ \oplus^* \circ \Mcal(\Omega_+(\mu) \cup \Acal) \bigr) \nabla_\pi \Jcal(\Gcal_{\gamma, \varepsilon}(\mu)) \\
		&\quad + \bigl( \nabla_{\mu_1} \Jcal(\Gcal_{\gamma, \varepsilon}(\mu))^\top, \zero^\top \bigr)
	\end{align*}
	is an element of $\partial \bigl( \Jcal \circ \Gcal_{\gamma, \varepsilon} \bigr)(\mu)$ as claimed.
\end{proof}

\begin{remark}
\label{rm:REDU-ReduTargetSubgrad}
	An earlier version of this manuscript included a purported proof of a result concerning the subdifferential of the reduced target function $f_{\gamma, \varepsilon}$ instead of the composition $F_{\gamma, \varepsilon}$. Upon closer inspection, and as pointed out by a reviewer, the argument relied on an implicit use of the restriction operator in a way that is not justified within the framework of subdifferential calculus. While the derivation of subgradients for the unrestricted composition remains valid, the extension to the restricted case lacks a rigorous foundation.
	
	Despite this, in our numerical experiments in \cref{sc:NUM}, we retain the use of a projection of $F_{\gamma, \varepsilon}$'s subgradients $g = (g_1, g_2) \in \partial F_{\gamma, \varepsilon}(\mu)$ at some point $\mu \in \R^{n_1} \times \R^{n_2}$ onto their first component $g_1$. This approach, while not theoretically justified in full generality, appears to yield correct and stable solutions in practice. The empirical success of this method suggests that the projected subgradients may still capture essential descent directions for the optimization problem at hand. We leave a rigorous justification of this observation as an open question for future work.
\end{remark}

\section{(Preliminary) Numerical Experiments}\label{sc:NUM}
The purpose of this section is not to present a sophisticated numerical scheme for solving the bilevel Hitchcock problems or to compare the performance of different algorithms in the context of these bilevel problems, but rather to validate the results of the previous sections and, in particular, to show that we can indeed approximate solutions to the non-regularized bilevel Hitchcock problems \cref{pr:PROB-BilevelHitch} by solutions to the reduced bilevel Hitchcock problems \cref{pr:REDU-ReduBilevelHitch} when we drive the regularization parameters $\gamma$ and $\varepsilon$ to zero.

To this end, let us consider a toy problem. We assume that there are a unknown source marginal $\mu_1^* \in \R^{n_1}$ as well as a known target marginal $\mu_2^\drm \in \R^{n_2}$ such that $\mu_1^*, \mu_2^\drm \geq 0$ and $\one^\top \mu_1^* = \one^\top \mu_2^\drm = 1$. Moreover, assume that the cost of transportation is given by some known cost matrix $c_\drm \in \R^{n_1 \times n_2}$. According to \cref{th:PROB-HitchSolu}, there is an optimal transportation plan $\pi^*$ between the marginals $\mu_1^*$ and $\mu_2^\drm$ \wrt the cost $c_\drm$, which we do not know in advance. However, we assume that we can observe both $\mu_1^*$ and $\pi^*$ on parts of their domains, namely $D_1 \subset \Omega_1$ and $D \subset \Omega$, respectively. Denote these observations by $\mu_1^\drm$ and $\pi_\drm$.

If we then choose, for some weighting parameter $\lambda > 0$, the tracking-type target function
\begin{equation}
\label{eq:NUM-TrackTypeObj}
	\Jcal(\pi, \mu_1)
	= \frac12 \norm[D]{\pi - \pi_\drm}^2 + \frac{\lambda}{2} \norm[D_1]{\mu_1 - \mu_1^\drm}^2,
\end{equation}
where the norms $\norm[D]{\substitute}$ and $\norm[D_1]{\substitute}$ are just the usual norms restricted to $D$ and $D_1$, respectively, then the bilevel Hitchcock problem \cref{pr:PROB-BilevelHitch} turns into the \emph{transportation identification problem}
\begin{equation*}
\label{pr:NUM-TransIdentProb}
\tag{TI}
	\begin{array}{rl}
		\inf\limits_{\pi, \mu_1} & \displaystyle \frac12 \norm[D]{\pi - \pi_\drm}^2 + \frac{\lambda}{2} \norm[D_1]{\mu_1 - \mu_1^\drm}^2 \\[1em]
		\text{s.t.} & \pi \in \R^{n_1 \times n_2}, \quad \mu_1 \in \R^{n_1}, \\[0.5em]
		& \mu_1 \geq 0, \quad \mu_1^\top \one = {\mu_2^\drm}^\top \one, \\[0.5em]
		& \pi ~ \text{solves \cref{pr:PROB-Hitch} w.r.t.} ~ \mu_1, ~ \mu_2^\drm, ~ \text{and} ~ c_\drm,
	\end{array}
\end{equation*}
which is the problem of reconstructing the unknown source marginal $\mu_1^*$ and the unknown optimal transport plan $\pi^*$ based on the (possibly error-prone) observations $\mu_1^\drm$ and $\pi_\drm$.

The benefits of this type of problem are obvious: if we consider a weighting parameter $\lambda > 0$, the observation domains $D_1 = \Omega_1$ and $D = \Omega$, and the observations $\mu_1^\drm = \mu_1^*$ and $\pi_\drm = \pi^*$, the point $(\pi^*, \mu_1^*)$, which realizes the target value $\Jcal(\pi^*, \mu_1^*) = 0$, is the unique solution to \cref{pr:NUM-TransIdentProb}. By fixing $\mu_1^*$ and $\pi^*$ in advance, we can test our results from the previous sections on a nontrivial bilevel problem whose (unique) solution is already known. \textcolor{cyan}{In particular, the point $(\pi^*, \mu_1^*)$ can be chosen arbitrarily (subject to feasibility)}. If, on the other hand, $D_1$ or $D$ are proper subsets of the domains or if $\mu_1^\drm$ or $\pi_\drm$ incorporate error terms, this allows us to introduce incomplete information or uncertainty to the problem.

We are going to solve the transportation identification problem \cref{pr:NUM-TransIdentProb} by the method we introduced in \cref{sc:REDU}, \ie we choose regularization parameters $\gamma, \varepsilon > 0$ and consider the \emph{reduced transportation identification problem}
\begin{equation*}
\label{pr:NUM-ReduTransIdentProb}
\tag{RTI$_\gamma^\varepsilon$}
	\begin{array}{rl}
		\inf\limits_{\mu_1} & \displaystyle f_{\gamma, \varepsilon}(\mu_1) \\[1em]
		\text{s.t.} & \mu_1 \in \StdSmplx,
	\end{array}
\end{equation*}
with the reduced target function
\begin{equation*}
	f_{\gamma, \varepsilon}(\mu_1) \coloneqq \frac12 \norm[D]{\Scal_{\gamma, \varepsilon}(\mu_1, \mu_2^\drm) - \pi_\drm}^2 + \frac{\lambda}{2} \norm[D_1]{\mu_1 - \mu_1^\drm}^2
\end{equation*}
and the feasible set
\begin{equation*}
	\StdSmplx \coloneqq \bigl\{
	v \in \R^{n_1}
	\colon v \geq 0, ~ v^\top \one = 1
	\bigr\},
\end{equation*}
which is just the standard simplex of $\R^{n_1}$. The tracking-type target function $\Jcal$ from \cref{eq:NUM-TrackTypeObj} is smooth \wrt $\pi$ and $\mu_1$. Consequently, $f_{\gamma, \varepsilon}$ is Lipschitz continuous \wrt $\mu_1$ and, for every point $\mu_1 \in \R^{n_1}$, we can compute an heuristic approximation of its Clarke subgradients, see \cref{pp:REDU-CompoDiff} and \cref{rm:REDU-ReduTargetSubgrad}.

That this approximation actually produces decent results, if we drive $\gamma$ and $\varepsilon$ towards $0$, will be shown in \cref{sc:RESU}. First, however, in \cref{sc:ALG}, we briefly discuss the method with which we solve the problems \cref{pr:NUM-ReduTransIdentProb}, $\gamma, \varepsilon > 0$.

% % % % % % % % % % % % % % % % % % % % % % % % % % % % % % % % % % % % % % %

\subsection{Algorithmic Implementation}\label{sc:ALG}
Because $f_{\gamma, \varepsilon}$ is Lipschitz continuous and bears Clarke subgradients at every point, we use the constrained nonsmooth trust region (TR) method from \cite{hillbrecht2024quadratic}, which originated from the (unconstrained) nonsmooth TR method proposed in \cite{christof2020nonsmooth}. The constrained nonsmooth TR method, which we present below in \cref{al:NUM-TrustRegionConstr}, was modified to be able to solve instances of the bilevel Hitchcock problem such as the the reduced transportation identification problem \cref{pr:NUM-ReduTransIdentProb}.

As already mentioned at the beginning of \cref{sc:NUM}, in this paper, we are only interested in an experimental validation of our results and the approximability of solutions to the bilevel Hitchcock problems. Therefore, we only present the constrained nonsmooth TR method applied to the reduced transportation identification problem as a reference and afterwards briefly comment on some details of the implementation. Note that the presented TR method is still subject to ongoing research. A more detailed (performance based) discussion of both the constrained and non-constrained nonsmooth TR methods can be found in \cite{hillbrecht2024quadratic} and \cite{christof2020nonsmooth}, respectively.

%\vspace*{1em}
%\begin{breakablealgorithm}
\begin{breakablealgorithm}{(a constrained nonsmooth TR method).}
\label{al:NUM-TrustRegionConstr}
%\caption{(a constrained nonsmooth TR method).}
	\begin{algorithmic}[1] %
		\State\label{it:Num-Init} \emph{Initialization:} Choose a model function $\phi \colon \R^{n_1} \times \R_+ \times \R^{n_1}$ in the sense of \cite[Assumption 6.9]{hillbrecht2024quadratic}. Moreover, choose the constants
		\begin{equation*}
			R, \Delta_{\text{min}} > 0, \quad
			0 < \eta_1 < \eta_2 < 1, \quad
			0 < \beta_1 < 1 < \beta_2, \quad
			0 < \nu \leq 1,
		\end{equation*}
		an initial point $\mu_{1, 0} \in \StdSmplx$, and an initial TR radius $\Delta_0 > \Delta_{\min}$. Set $k \leftarrow 0$.
		
		\For{$k = 0, 1, 2, \dots$}
		
		\State\label{it:NUM-gk} Compute both a (Clarke) subgradient $g_k \in \partial f_{\gamma, \varepsilon}(\mu_{1,k})$ and a symmetric matrix $H_k \in \R_{\textrm{sym}}^{n_1 \times n_1}$.
		
		\If{\label{it:NUM-StatMeas}$\theta_R(\mu_{1,k}, g_k) = 0$, with the stationarity measure $\theta_R$ being defined by
			\begin{equation*}
				\theta_R(\mu_{1,k}, g_k) \coloneqq - \min_{d \in \StdSmplx - \mu_{1,k}, \norm{d} \leq R} \scalarproduct{g_k}{d}
				\quad \geq 0,
			\end{equation*}
		}
		
		\State\label{it:NUM-Stop} \textbf{stop:} $\mu_{1,k}$ satisfies the generalized variational inequality
		\begin{equation}
		\label{eq:NUM-VI}
		\tag{VI}
			f_{\gamma, \varepsilon}^\circ(\mu_{1,k}; z - \mu_{1,k}) \geq 0
			\quad \text{for all} ~ z \in \StdSmplx,
		\end{equation}
		where $f_{\gamma, \varepsilon}^\circ$ denotes Clarke's generalized directional derivative.
		
		\Else
		
		\If{$\Delta_k \geq \Delta_{\text{min}}$}
		
		\State\label{it:NUM-SoluSubprob} Compute an (inexact) solution $d_k$ of the constrained TR subproblem
		\begin{equation}
		\label{pr:NUM-ConstTRSubpr}
			\tag{Q$_k$}
			\begin{array}{rl}
				\inf\limits_{d} & q_k(d) \coloneqq f_{\gamma, \varepsilon}(\mu_{1,k}) + \scalarproduct{g_k}{d} + \frac12 d^\top H_k d \\[0.5em]
				\text{s.t.} & d \in \StdSmplx - \mu_{1,k}, ~ \norm{d} \leq \Delta_k
			\end{array}
		\end{equation}
		that satisfies the constrained Cauchy decrease condition
		\begin{align*}
			&f_{\gamma, \varepsilon}(\mu_{1,k}) - q_k(d_k) \\
			&\quad \geq \frac\nu{2R} \theta_R(\mu_{1,k}, g_k) \min \Bigl\{ R, \Delta_k, \frac{\theta_R(\mu_{1,k}, g_k)}{R \norm{H_k}} \Bigr\}.
		\end{align*}
		
		\State Compute the quality indicator
		\begin{equation*}
			\rho_k \coloneqq \frac{f_{\gamma, \varepsilon}(\mu_{1,k}) - f_{\gamma, \varepsilon}(\mu_{1,k} + d_k)}{f_{\gamma, \varepsilon}(\mu_{1,k}) - q_k(d_k)}.
		\end{equation*}
		
		\Else
		
		\State\label{it:NUM-SoluSubprobModi} Compute an (inexact) solution $\tilde{d}_k$ of the modified constrained TR subproblem
		\begin{equation}
		\label{pr:NUM-ModConstTRSubpr}
			\tag{$\tilde{\text{Q}}_k$}
			\begin{array}{rl}
				\inf\limits_{d}
				& \tilde{q}_k(d) \coloneqq f_{\gamma, \varepsilon}(\mu_{1,k}) + \phi(\mu_{1,k}, \Delta_k; d) + \frac{1}{2} \, d^\top H_k d \\[0.5em]
				\text{s.t.} & d \in \StdSmplx - \mu_{1,k}, ~ \norm{d} \leq \Delta_k,
			\end{array}
		\end{equation}
		that satisfies the modified constrained Cauchy decrease condition
		\begin{align*}
			&f_{\gamma, \varepsilon}(\mu_{1,k}) - \tilde{q}_k(\tilde{d}_k) \\
			&\quad \geq \frac{\nu}{2 R} \psi_R(\mu_{1,k}, \Delta_k) \min \Bigl\{ R, \Delta_k, \frac{\psi_R(\mu_{1,k}, \Delta_k)}{R \norm{H_k}} \Bigr\},
		\end{align*}
		with the modified stationarity measure $\psi_R$ being defined by
		\begin{equation*}
			\psi_R(\mu_{1,k}, \Delta_k)
			\coloneqq - \min_{d \in \StdSmplx - \mu_{1,k}, \norm{d} \leq R} \phi(\mu_{1,k}, \Delta_k, d).
		\end{equation*}
		
		\State Compute the modified quality indicator
		\begin{equation*}
			\rho_k
			\leftarrow \begin{cases}
				\displaystyle{\frac{f_{\gamma, \varepsilon}(\mu_{1,k}) - f_{\gamma, \varepsilon}(\mu_{1,k} + d_k)}{f_{\gamma, \varepsilon}(\mu_{1,k}) - \tilde{q}_k(d_k)}},
				& \begin{array}{l}
					\text{if} ~ \psi_R(\mu_{1,k}, \Delta_k) \\
					\quad > \theta_R(\mu_{1,k}, g_k) \Delta_k,
				\end{array} \\
				0, & \begin{array}{l}
					\text{if} ~ \psi_R(\mu_{1,k}, \Delta_k) \\
					\quad \leq \theta_R(\mu_{1,k}, g_k) \Delta_k.
				\end{array}
			\end{cases}
		\end{equation*}
		
		\EndIf
		
		\State\label{it:NUM-Update} \textbf{update:} Set
		\begin{equation*}
		\mu_{1, k+1}
			\leftarrow \begin{cases}
				\mu_{1,k}, & \text{if} ~ \rho_k \leq \eta_1 \\
				\mu_{1,k} + d_k, & \text{if} ~ \rho_k > \eta_1,
			\end{cases}
		\end{equation*}
		and
		\begin{equation*}
			\Delta_{k+1}
			\leftarrow \begin{cases}
				\beta_1 \Delta_k, & \text{if} ~ \rho_k \leq \eta_1, \\
				\max\{\Delta_{\min}, \Delta_k\}, & \text{if} ~ \eta_1 < \rho_k \leq \eta_2, \\
				\max\{\Delta_{\min}, \beta_2 \Delta_k\}, & \text{if} ~ \rho_k > \eta_2.
			\end{cases}
		\end{equation*}
		Set $k \leftarrow k+1$.
		
		\EndIf
		
		\EndFor
	\end{algorithmic}
\end{breakablealgorithm}
The presented algorithm gives rise to several remarks.
\begin{remark}
	\begin{itemize}
		\item As a model function in {\hyperref[it:Num-Init]{Step}\ \ref{it:Num-Init}}, we choose the function
		\begin{equation}
			\label{eq:NUM-ModelFunc}
			\phi(\mu_1, \Delta; d)
			\coloneqq \sup_{G \in \Gcal((\mu_1,\mu_2^\drm), \Delta)} \scalarproduct{p_G + \nabla_{\mu_1} \Jcal \bigl( \Scal_{\gamma, \varepsilon}(\mu_1, \mu_2^\drm), \mu_1 \bigr)}{d},
		\end{equation}
		where
		\begin{equation*}
			\Gcal(\mu, \Delta)
			\coloneqq \bigcup_{\xi \in \overline{B(\mu; \Delta)}} \partial_B \Scal_{\gamma, \varepsilon}(\xi)
		\end{equation*}
		denotes the collective Bouligand subdifferential, which collects all Bouligand subgradients of $\Scal_{\gamma, \varepsilon}$ in a ball around a given point, and $p_G$ corresponds to the first element of the tuple $G^* \nabla_\pi \Jcal \bigl( \Scal_{\gamma, \varepsilon}(\mu_1, \mu_2^\drm), \mu_1 \bigr) \in \R^{n_1} \times \R^{n_2}$, see \cref{pp:REDU-CompoDiff}. In some sense, the purpose of the model function $\phi$ is to collect first-order information in the vicinity of the current iterate to prohibit convergence to nonstationary points.
		
		Whether the construction of the model function from \cref{eq:NUM-ModelFunc} meets all the requirements specified in \cite[Assumption 6.9]{hillbrecht2024quadratic} is currently an open question. However, it can be proven that the collective Bouligand subdifferential $\Gcal$ satisfies the properties outlined in \cite[Assumption 4.1]{christof2020nonsmooth}, see \cite[Lemma 6.12]{hillbrecht2024quadratic}. In the unconstrained case, these properties are sufficient for the model function to meet requirements \textcolor{cyan}{which} are \textcolor{cyan}{just} the non-constrained counterparts of \cite[Assumption 6.9]{hillbrecht2024quadratic}. Therefore, it seems reasonable to adopt the same model function for the constrained case of \cref{pr:NUM-ReduTransIdentProb}.
		
		\item We compute the matrix $H_k$ in {\hyperref[it:NUM-gk]{Step}\ \ref{it:NUM-gk}} via BFGS update formula.
		
		\item If one chooses $R > 0$ large enough, then due to the structure of the standard simplex $\StdSmplx$ the calculation of the stationarity measure in {\hyperref[it:NUM-StatMeas]{Step}\ \ref{it:NUM-StatMeas}} reduces to solving a linear problem.
		
		\item The stopping criteria from {\hyperref[it:NUM-Stop]{Step}\ \ref{it:NUM-Stop}} of \cref{al:NUM-TrustRegionConstr} is just a necessary condition for local minima of the constrained optimization problem \cref{pr:NUM-ReduTransIdentProb}. This immediately follows from the definition of Clark's generalized directional derivative.
		
		\item In {\hyperref[it:NUM-SoluSubprob]{Step}\ \ref{it:NUM-SoluSubprob}}, we obtain an inexact solution of the constrained TR subproblem \cref{pr:NUM-ConstTRSubpr} by computing a minimizing convex combination of
		\begin{enumerate}
			\item the direction that realizes the minimum in the calculation of the stationarity measure in {\hyperref[it:NUM-StatMeas]{Step}\ \ref{it:NUM-StatMeas}}, the latter of which can be seen as a linearization of \cref{pr:NUM-ConstTRSubpr};
			
			\item the projection of the dogleg step, which corresponds to the TR subproblem without the linear constraints, onto the standard simplex.
		\end{enumerate}
		
		\item It is currently not clear whether there exists a manageable representation of the (possibly uncountable) collective Bouligand subdifferential $\Gcal$ from above. Therefore, we cannot compute neither a global solution of the modified constrained TR subproblem \cref{pr:NUM-ModConstTRSubpr} nor the modified stationarity measure $\psi_R$ from {\hyperref[it:NUM-SoluSubprobModi]{Step}\ \ref{it:NUM-SoluSubprobModi}} exactly, but have to rely on an approximations thereof. We obtain this approximations by iteratively exploring the ball around the current iterate and collecting the corresponding Bouligand subgradients to find an approximation of the collective Bouligand subdifferential which we then use to approximate the model function at a given point. (We know these are a lot of approximations, but the modified constrained TR subproblem is only supposed to act as a ``safeguard''.)
	\end{itemize}
\end{remark}
The authors provide an actual implementation of this algorithm on GitHub: \url{https://github.com/sebastianhillbrecht/cntr_method}. Note, however, that the implementation of {\hyperref[it:NUM-SoluSubprobModi]{Step}\ \ref{it:NUM-SoluSubprobModi}} is only approximate and therefore offers no theoretical guarantee of convergence.

% % % % % % % % % % % % % % % % % % % % % % % % % % % % % % % % % % % % % % %

\subsection{Results of the Numerical Experiments}\label{sc:RESU}
For the first numerical experiment in the framework of the transportation identification problem \cref{pr:NUM-TransIdentProb}, we set $n_1 = n_2 = 25$ and choose random marginals $\mu_1^*, \mu_2^\drm \in \R^{n_1}$, which are nonnegative, occupied to roughly $50\%$, and sum to $1$. \textcolor{cyan}{We then} compute an optimal transport plan $\pi^*$ which is transporting $\mu_1^*$ onto $\mu_2^\drm$ \wrt the cost given by $c_\drm(i_1, i_2) = |i_1 - i_2|^2$. The resulting variables are shown in \cref{fg:NUM-InitialData}.
\begin{figure}
	\centering
	\begin{subfigure}{0.32\textwidth}
		\includegraphics[width=\textwidth]{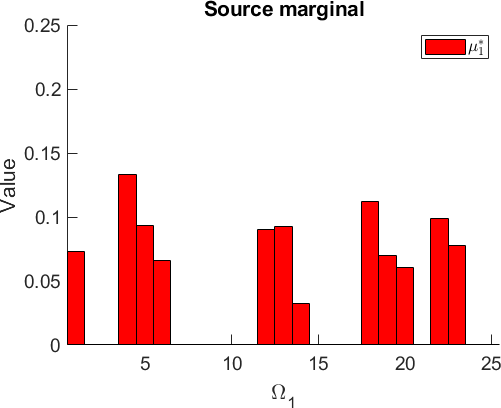}
		\caption{}
	\end{subfigure}
	\hfill
	\begin{subfigure}{0.25\textwidth}
		\includegraphics[width=\textwidth]{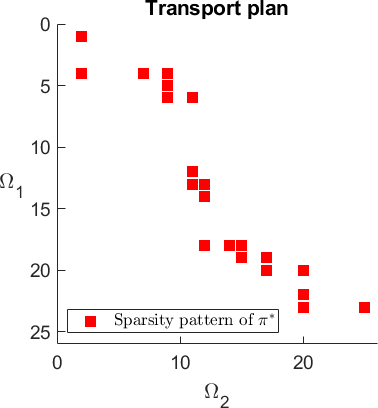}
		\caption{}
		\label{fg:NUM-PiStar}
	\end{subfigure}
	\hfill
	\begin{subfigure}{0.32\textwidth}
		\includegraphics[width=\textwidth]{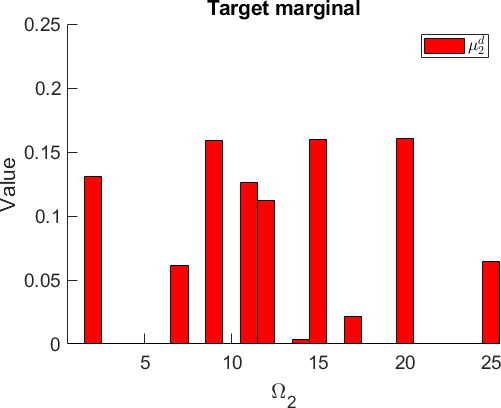}
		\caption{}
	\end{subfigure}
	\caption{Randomly generated (semi-sparse) data. \cref{fg:NUM-PiStar} shows the sparsity pattern of the optimal transport plan $\pi^*$.}
	\label{fg:NUM-InitialData}
\end{figure}
We then choose the observation domains $D_1 = \Omega_1$ and $D = \Omega$, the exact observations $\mu_1^\drm = \mu_1^*$ and $\pi_\drm = \pi^*$, as well as the weight $\lambda = 1$. As already mentioned, in this setting, the unique solution of the transportation identification problem \cref{pr:NUM-TransIdentProb} is given by the couple $(\pi^*, \mu_1^*)$.

\cref{fg:NUM-Test1Mu1Approx} shows the evolution of the cluster point $\bar{\mu}_1$ of \cref{al:NUM-TrustRegionConstr} that we applied to solve the reduced transportation identification problem \cref{pr:NUM-ReduTransIdentProb} and \cref{fg:NUM-Test1PiApprox} shows the corresponding optimal transport plan $\bar{\pi}$ for different choices of the regularization parameters $\gamma$ and $\varepsilon$. For the constrained nonsmooth TR method, we chose the standard parameter configuration $R = \sqrt{n_1}$, $\Delta_{\min} = 10^{-6}$, $\eta_1 = 0.1$, $\eta_2 = 0.9$, $\beta_1 = 0.5$, $\beta_2 = 1.5$, and $\nu = 1$. The initial point and the initial TR radius were set to be $\mu_{1, 0} = {n_1}^{-1} \one$ and $\Delta_0 = 1$, respectively, for every application of the method. We set the stationarity tolerance for the termination criterion in {\hyperref[it:NUM-Stop]{Step}\ \ref{it:NUM-Stop}} to $\verb*|TOL| = 10^{-5}$. This tolerance was achieved after a maximum of $70$ iterations in each test run shown.
\begin{figure}
	\centering
	\begin{subfigure}{0.32\textwidth}
		\includegraphics[width=\textwidth]{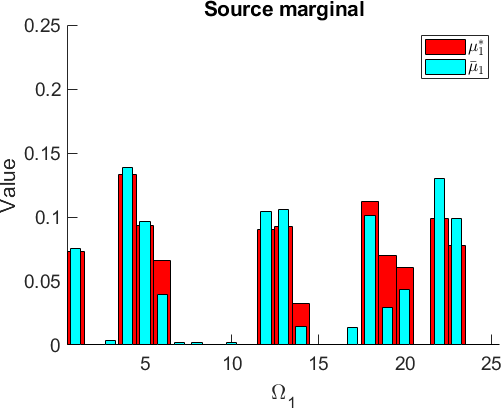}
		\caption{$\gamma = \varepsilon = 10^{-2}$}
		\label{fg:NUM-Test1Mu1Approx-1}
	\end{subfigure}
	\hfill
	\begin{subfigure}{0.32\textwidth}
		\includegraphics[width=\textwidth]{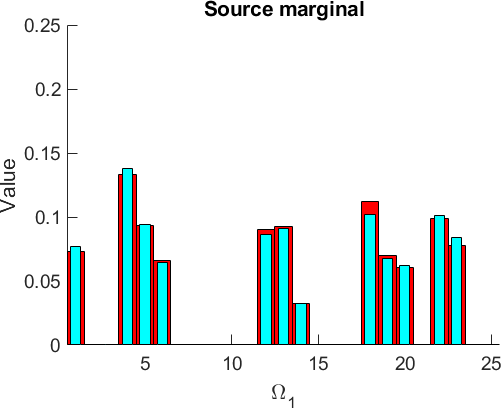}
		\caption{$\gamma = \varepsilon = 10^{-4}$}
		\label{fg:NUM-Test1Mu1Approx-4}
	\end{subfigure}
	\hfill
	\begin{subfigure}{0.32\textwidth}
		\includegraphics[width=\textwidth]{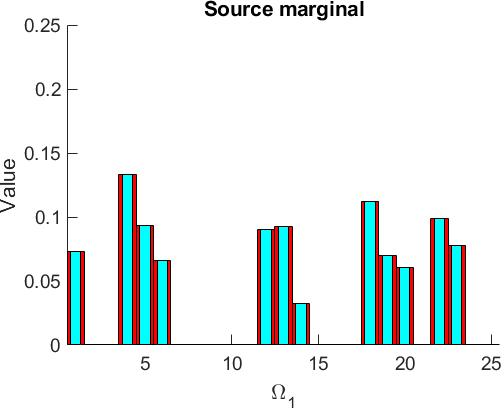}
		\caption{$\gamma = \varepsilon = 10^{-6}$}
		\label{fg:NUM-Test1Mu1Approx-6}
	\end{subfigure}
	\caption{Estimated source marginal $\bar{\mu}_1$ (blue bars) compared to the ``true'' source marginal $\mu_1^*$ (red bars) for different values of regularization parameters $\gamma$ and $\varepsilon$.}
	\label{fg:NUM-Test1Mu1Approx}
\end{figure}
\begin{figure}
	\centering
	\begin{subfigure}{0.3\textwidth}
		\includegraphics[width=\textwidth]{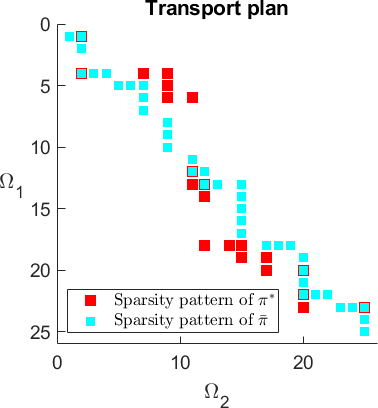}
		\caption{$\gamma = \varepsilon = 10^{-2}$}
	\end{subfigure}
	\hfill
	\begin{subfigure}{0.3\textwidth}
		\includegraphics[width=\textwidth]{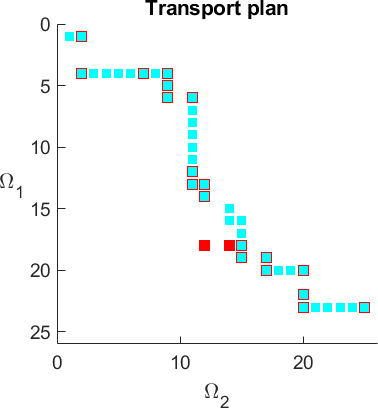}
		\caption{$\gamma = \varepsilon = 10^{-4}$}
	\end{subfigure}
	\hfill
	\begin{subfigure}{0.3\textwidth}
		\includegraphics[width=\textwidth]{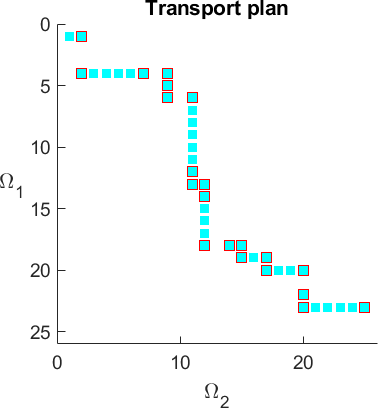}
		\caption{$\gamma = \varepsilon = 10^{-6}$}
	\end{subfigure}
	\caption{Comparison of the sparsity pattern (blue squares) of the optimal transport plans $\bar{\pi}$ corresponding to the estimated source marginals $\bar{\mu}_1$ from \cref{fg:NUM-Test1Mu1Approx} with the sparsity pattern (red squares) of the ``true'' optimal transport plan $\pi^*$ for different values of regularization parameters $\gamma$ and $\varepsilon$.}
	\label{fg:NUM-Test1PiApprox}
\end{figure}

We observe that even with relatively large regularization parameters (\ie $\gamma = \varepsilon = 10^{-2})$ the source marginal $\mu_1^*$ is reasonably approximated, see \cref{fg:NUM-Test1Mu1Approx-1}, and the quality of approximation becomes even better for declining regularization parameters, see \cref{fg:NUM-Test1Mu1Approx-4} -- \ref{fg:NUM-Test1Mu1Approx-6}. When examining the corresponding optimal transport plans, it can be seen that the approximation is inaccurate for larger regularization parameters, but improves significantly up to a point where the sparsity pattern of $\pi^*$ is completely captured, see \cref{fg:NUM-Test1PiApprox}. This (visual) observation is underpinned by the data given in \cref{tb:NUM-Test1Data}.
\begin{table}
	\centering
	\begin{tabular}{c|c|c|c}
		$\gamma = \varepsilon$ & $\max(|\mu_1^* - \bar{\mu}_1|)$ & $\max(|\pi^* - \bar{\pi}|)$ & $\Jcal(\bar{\pi}, \bar{\mu}_1)$ \\[0.5ex]
		\hline
		\vphantom{$M^{M^M}$}$10^{-2}$ & $3.9855 \cdot 10^{-2}$ & $9.9722 \cdot 10^{-2}$ & $2.9307 \cdot 10^{-2}$ \\
		$10^{-4}$ & $1.0519 \cdot 10^{-2}$ & $8.9377 \cdot 10^{-2}$ & $2.6759 \cdot 10^{-4}$ \\
		$10^{-6}$ & $1.2200 \cdot 10^{-4}$ & $7.1756 \cdot 10^{-5}$ & $3.4216 \cdot 10^{-8}$
	\end{tabular}
	\caption{Collection of end-of-iteration data of the first numerical experiment for different values of regularization parameters $\gamma$ and $\varepsilon$.}
	\label{tb:NUM-Test1Data}
\end{table}

For the second experiment, we reuse the data (\ie the marginals, the cost matrix, and the optimal transport plan) from the first experiment but now consider different observation domains. In particular, we set $D_1 = \{9, \dots, 15\}$ and define $D$ to correspond to a band matrix with upper and lower bandwidth of $3$. The observed variables $\mu_1^\drm$ and $\pi_\drm$ are defined to be the restrictions of $\mu_1^*$ and $\pi^*$ to $D_1$ and $D$, respectively. Again, $(\pi^*, \mu_1^*)$ is a solution to the corresponding transportation identification problem \cref{pr:NUM-TransIdentProb}.

We again use the standard parameter configuration of the TR method. Similarly to before, \cref{fg:NUM-Test2Mu1Approx} shows the evolution of the cluster point $\bar{\mu}_1$ and \cref{fg:NUM-Test2PiApprox} shows the corresponding optimal transport plan $\bar{\pi}$ for different choices of the regularization parameters $\gamma$ and $\varepsilon$. In contrast to the previous experiment, the TR method exceeded the iteration limit of $200$ iterations in two of the three tests presented.
\begin{figure}
	\centering
	\begin{subfigure}{0.3\textwidth}
		\includegraphics[width=\textwidth]{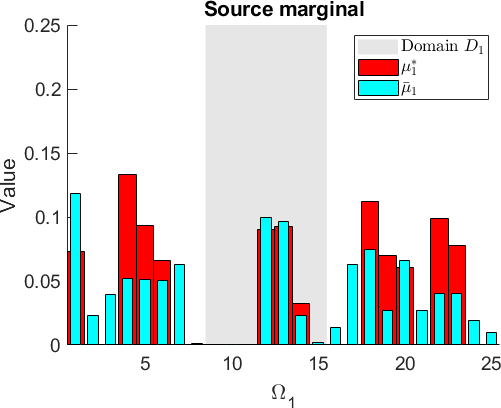}
		\caption{$\gamma = \varepsilon = 10^{-2}$}
	\end{subfigure}
	\hfill
	\begin{subfigure}{0.3\textwidth}
		\includegraphics[width=\textwidth]{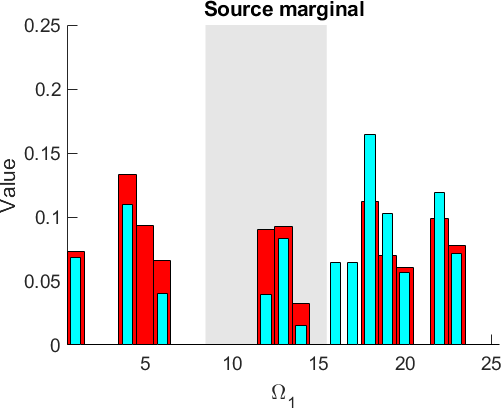}
		\caption{$\gamma = \varepsilon = 10^{-4}$}
	\end{subfigure}
	\hfill
	\begin{subfigure}{0.3\textwidth}
		\includegraphics[width=\textwidth]{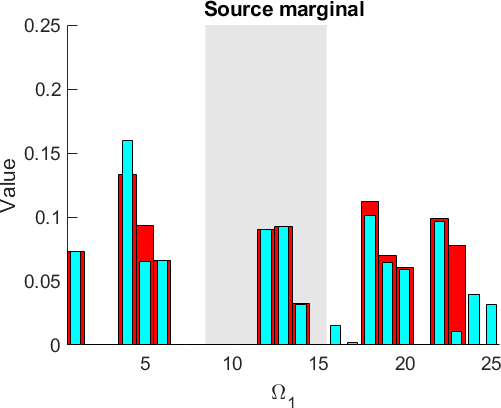}
		\caption{$\gamma = \varepsilon = 10^{-6}$}
	\end{subfigure}
	\caption{Estimated source marginal $\bar{\mu}_1$ (blue bars) compared to the ``true'' source marginal $\mu_1^*$ (red bars) for different values of regularization parameters $\gamma$ and $\varepsilon$. The gray region indicates the observation domain $D_1$.}
	\label{fg:NUM-Test2Mu1Approx}
\end{figure}
\begin{figure}
	\centering
	\begin{subfigure}{0.3\textwidth}
		\includegraphics[width=\textwidth]{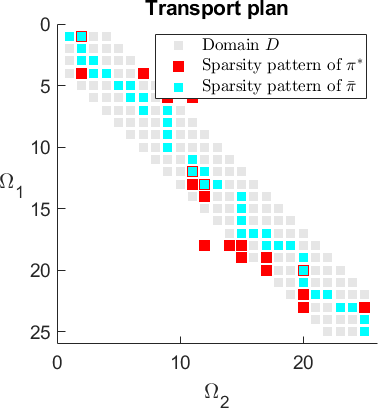}
		\caption{$\gamma = \varepsilon = 10^{-2}$}
	\end{subfigure}
	\hfill
	\begin{subfigure}{0.3\textwidth}
		\includegraphics[width=\textwidth]{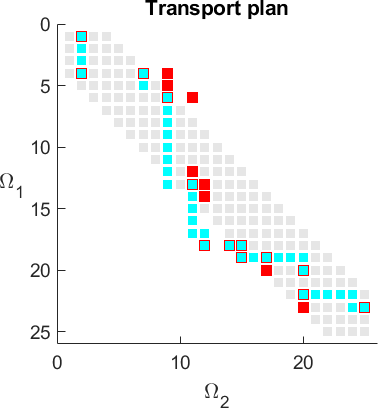}
		\caption{$\gamma = \varepsilon = 10^{-4}$}
	\end{subfigure}
	\hfill
	\begin{subfigure}{0.3\textwidth}
		\includegraphics[width=\textwidth]{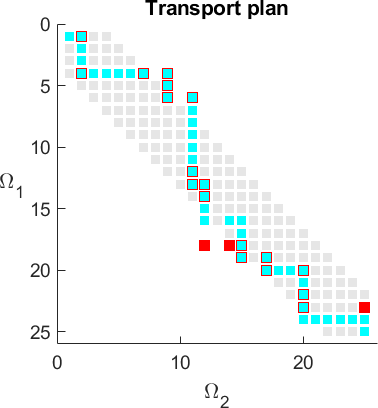}
		\caption{$\gamma = \varepsilon = 10^{-6}$}
	\end{subfigure}
	\caption{Comparison of the sparsity pattern (blue squares) of the optimal transport plans $\bar{\pi}$ corresponding to the estimated source marginals $\bar{\mu}_1$ from \cref{fg:NUM-Test1Mu1Approx} with the sparsity pattern (red squares) of the ``true'' optimal transport plan $\pi^*$ for different values of regularization parameters $\gamma$ and $\varepsilon$. The gray region indicates the observation domain $D$.}
	\label{fg:NUM-Test2PiApprox}
\end{figure}

Again, we find that the quality of the approximation of both the source marginal and the corresponding optimal transport plan increases when the regularization parameters are reduced, see \cref{tb:NUM-Test2Data}. Moreover, it seems that we can even (to some extent) approximate both variables outside the observation domain. We suspect that this behavior is due to the fact that the support of the transport plan $\pi^*$ lies to a large extent in the observation domain $D$ and that the relationship between marginals and transport plan is continuous. However, if we compare the objective function values of the two experiments, see the last columns of \cref{tb:NUM-Test1Data} and \cref{tb:NUM-Test2Data}, we find that the quality of the approximation is several orders of magnitude worse in the latter case. However, this is not surprising since in the first experiment we had complete information (encoded in the objective function and its derivatives) about the source marginal and the optimal transportation plan, while in the second experiment there was a great lack of information about the source marginal.
\begin{table}
	\centering
	\begin{tabular}{c|c|c|c}
		$\gamma = \varepsilon$ & $\max(|\mu_1^* - \bar{\mu}_1|)$ & $\max(|\pi^* - \bar{\pi}|)$ & $\Jcal(\bar{\pi}, \bar{\mu}_1)$ \\[0.5ex]
		\hline
		\vphantom{$M^{M^M}$}$10^{-2}$ & $8.1087 \cdot 10^{-2}$ & $9.9722 \cdot 10^{-2}$ & $2.7878 \cdot 10^{-2}$ \\
		$10^{-4}$ & $9.3342 \cdot 10^{-2}$ & $9.3342 \cdot 10^{-2}$ & $1.0152 \cdot 10^{-2}$ \\
		$10^{-6}$ & $6.6571 \cdot 10^{-2}$ & $6.4498 \cdot 10^{-2}$ & $3.1405 \cdot 10^{-3}$
	\end{tabular}
	\caption{Collection of end-of-iteration data of the second numerical experiment for different values of regularization parameters $\gamma$ and $\varepsilon$.}
	\label{tb:NUM-Test2Data}
\end{table}

% % % % % % % % % % % % % % % % % % % % % % % % % % % % % % % % % % % % % % %

\section*{Acknowledgements}

The author would like to thank Prof.\ Gerd Wachsmuth for valuable comments and discussions on an earlier version of this manuscript. In particular, the author is grateful for an observation that led to a considerably shorter proof of the existence of a recovery sequence.

\bibliographystyle{jnsao.bst}
\bibliography{refs.bib}

\end{document}